\documentclass[10.5pt,a4paper]{article}

\usepackage{graphicx,latexsym,euscript,makeidx,color,bm}
\usepackage{amsmath,amsfonts,amssymb,amsthm,thmtools,mathrsfs,enumerate}
\usepackage[colorlinks,linkcolor=blue,anchorcolor=green,citecolor=red]{hyperref}

\usepackage{geometry}
   \def\cB{{\cal B}}  
     
\def\dbD{\mathbb{D}}     
\def\dbE{\mathbb{E}}     
\def\dbF{\mathbb{F}}   \def\cF{{\cal F}}  
     
\def\dbH{\mathbb{H}}   \def\cH{{\cal H}}

\def\dbP{\mathbb{P}}     
     
\def\dbR{\mathbb{R}}

      \def\lt{\left}       \def\hb{\hbox}
\def\ms{\medskip}        \def\rt{\right}      \def\ae{\hbox{\rm a.e.}}
\def\bs{\bigskip}        \def\lan{\langle}    \def\as{\hbox{\rm a.s.}}
\def\ds{\displaystyle}   \def\ran{\rangle}    \def\tr{\hbox{\rm tr$\,$}}
\def\ts{\textstyle}         
\def\no{\noindent}

\def\rf{\eqref}            \def\hp{\hphantom}
\def\deq{\triangleq}     \def\({\Big (}       \def\nn{\nonumber}
\def\les{\leqslant}      \def\){\Big )}       
      \def\[{\Big[}        \def\cl{\overline}
\def\ti{\tilde}          \def\]{\Big]}        
      \def\q{\quad}        
         \def\qq{\qquad}      \def\1n{\negthinspace}
\def\cd{\cdot}           \def\2n{\1n\1n}      \def\3n{\1n\1n\1n}

\def\a{\alpha}              \def\Om{\Omega}   \def\om{\omega}
         \def\D{\Delta}           
\def\z{\zeta}         \def\Th{\Theta}  \def\th{\theta}    \def\si{\sigma}
\def\e{\varepsilon}     \def\l{\lambda}        
    \def\t{\tau}     \def\f{\varphi}  \def\i{\infty}   

\def\ba{\begin{array}}                \def\ea{\end{array}}
\def\bel{\begin{equation}\label}      \def\ee{\end{equation}}

\newtheorem{theorem}{Theorem}[section]

\newtheorem{proposition}[theorem]{Proposition}
\newtheorem{corollary}[theorem]{Corollary}
\newtheorem{lemma}[theorem]{Lemma}
\newtheorem{remark}[theorem]{Remark}

\newtheorem{taggedassumptionx}{}
\newenvironment{taggedassumption}[1]
 {\renewcommand\thetaggedassumptionx{#1}\taggedassumptionx}
 {\endtaggedassumptionx}
\makeatletter
   
   \@addtoreset{equation}{section}
\makeatother
  \allowdisplaybreaks[4]

\begin{document}

\title{\bf Extended Backward Stochastic Volterra Integral Equations,
        Quasilinear Parabolic Equations, and Feynman-Kac Formula}
\author{
        Hanxiao Wang\thanks{School of Mathematical Sciences, Fudan University,
                    Shanghai 200433, China (Email: {\tt hxwang14@} {\tt fudan.edu.cn}). This author is supported in part by the China Scholarship Council, while visiting University of Central Florida.}}
\maketitle

\no\bf Abstract. \rm
In this paper, we establish the relationship between backward stochastic Volterra integral equations (BSVIEs, for short)
and a kind of non-local quasilinear (and possibly degenerate) parabolic equations.
We first  introduce  the  extended  backward stochastic  Volterra  integral  equations  (EBSVIEs,  for  short).   Under  some  mild  conditions,  we establish the well-posedness of EBSVIEs and
obtain some regularity results of the adapted solution to the EBSVIEs via Malliavin calculus.
We show that a given function expressed in terms of the solution to the EBSVIEs  solves
a certain system of non-local parabolic partial differential equations (PDEs, for short),
which generalizes the famous nonlinear Feynman-Kac formula in Pardoux--Peng \cite{Pardoux--Peng 1992}.

\ms

\no\bf Keywords. \rm
Backward stochastic Volterra integral equation,
probabilistic representation, nonlinear Feynman-Kac formula,
time-inconsistent, quasilinear parabolic partial differential equations.

\ms

\no\bf AMS subject classifications. \rm 60H20, 45D05, 35K40, 35k59.
\section{Introduction}

Let $(\Om,\cF,\dbP)$ be a complete probability space on which a $d$-dimensional Brownian motion $W=\{W(t);0\les t<\i\}$ is defined,
with $\dbF=\{\cF_t\}_{t\geq 0}$ being the natural filtration of $W$ augmented by all the $\dbP$-null sets in $\cF$.
In this paper, we consider the following stochastic integral equation in $\dbR^m$,
\bel{ebsvie-I}
Y(t,s)=\psi(t) + \int_s^T g(t,r,Y(t,r),Y(r,r),Z(t,r))dr-\int_s^T Z(t,r)dW(r).
\ee
We call \rf{ebsvie-I} an {\it extended backward stochastic Volterra integral equation} (EBSVIE, for short).
By an {\it adapted solution} to \rf{ebsvie-I}, we mean a pair of $\dbR^m\times\dbR^{m\times d}$-valued random
fields $(Y(\cd,\cd),Z(\cd,\cd))=\{(Y(t,s),Z(t,s));0\les t, s\les T\}$ such that
\begin{enumerate}[(i)]
\item for each fixed $0\les t\les T$, $Y(t,\cd)$ is $\dbF$-progressively measurable and continuous,

\item for each fixed $0\les t\les T$, $Z(t,\cd)$ is $\dbF$-progressively measurable, and

\item \rf{ebsvie-I} is satisfied in the usual It\^{o} sense for Lebesgue-almost every $t\in[0,T]$.
\end{enumerate}
Here, $\dbR^m$ is the usual $m$-dimensional Euclidean space consisting of all $m$-tuple of real numbers, and $\dbR^{m\times d}$ is the set of all $m\times d$ real matrices.
It is noteworthy that condition  (i)  implies  that  $Y(r,r);0\les r\les T$ is well-defined and $\dbF$-progressively measurable.
In \rf{ebsvie-I}, $g$ and $\psi$ are called the {\it generator} and the {\it free term}, respectively.

\ms
Let us look at some special cases of EBSVIE \rf{ebsvie-I}.
Suppose
$$g(t,s,y,y',z)=g(t,s,y,z),\q\forall (t,s,y,y',z)\in[0,T]^2\times \dbR^m\times\dbR^m\times\dbR^{m\times d},$$
then EBSVIE \rf{ebsvie-I} is reduced to the following form:
\bel{ebsde} Y(t,s)=\psi(t) + \int_s^T g(t,r,Y(t,r),Z(t,r))dr - \int_s^T Z(t,r)dW(r),\ee
which is a family of so-called {\it backward stochastic differential equations} (BSDEs, for short) parameterized by $t\in[0,T]$;
see  \cite{Pardoux--Peng 1990, Karoui--Peng--Quenez 1997, Ma--Yong 1999, Zhang 2017} for systematic discussions of BSDEs.

\ms
On the other hand, if
$$g(t,s,y,y',z)=g(t,s,y',z),\q\forall(t,s,y,y',z)\in[0,T]^2\times \dbR^m\times\dbR^m\times\dbR^{m\times d},$$
let $s=t$ and $Y(t)=Y(t,t)$, then  EBSVIE \rf{ebsvie-I} is reduced to the following form:
\bel{bsvie} Y(t)=\psi(t) + \int_t^T g(t,r,Y(r),Z(t,r))dr - \int_t^T Z(t,r)dW(r), \ee
which is a so-called {\it backward stochastic Volterra integral equation}
(BSVIE, for short).
This is exactly why we call \rf{ebsvie-I} an extended backward stochastic Volterra integral equation.
BSVIEs of the form \rf{bsvie} was initially studied by Lin \cite{Lin 2002} and followed by several other researchers:
Aman and N'Zi \cite{Aman-N'Zi 2005}, Yong \cite{Yong 2007},  Ren \cite{Ren 2010}, Anh, Grecksch, and Yong \cite{Anh-Grecksch-Yong 2011},
Djordjevi'c and Jankovi'c \cite{Djordjevic-Jankovic 2013,Djordjevic-Jankovic 2015},
Hu and {\O}ksendal \cite{Hu 2018}, and the references therein.
Recently, Wang, Sun, and Yong \cite{Wang--Sun--Yong 2018} established the well-posedness of quadratic BSVIEs
(which means the generator $g(t,s,y,z)$ of \rf{bsvie} has a quadratic growth in $z$) and explored the applications of quadratic BSVIEs to equilibrium dynamic risk measure and equilibrium recursive utility process.

\ms
BSVIE of the more general form
\bel{bsvie-II} Y(t)=\psi(t) + \int_t^T g(t,r,Y(r),Z(t,r),Z(r,t))dr -\int_t^T Z(t,r)dW(r) \ee
was firstly introduced by Yong \cite{Yong 2008} in his research on optimal control of forward stochastic Volterra integral equations (FSVIEs, for short).
The BSVIE \rf{bsvie-II} has a remarkable feature that its solution might not be unique
due to lack of restriction on the term $Z(r,t);0\leq t\leq r \leq T$.
Suggested by the nature of the equation from the adjoint equation in the Pontryagin type maximum principle, Yong \cite{Yong 2008}
introduced the notion of {\it adapted M-solution}: A pair $(Y(\cd),Z(\cd\,,\cd))$ is called an adapted M-solution to \rf{bsvie-II}, if in addition to (i)--(iii) stated above, the following condition is also satisfied:
\bel{M-solution} Y(t)=\dbE[Y(t)]+\int_0^tZ(t,s)dW(s),\q\ae~t\in[0,T],~\as \ee
Under usual Lipschitz conditions, well-posedness was established in \cite{Yong 2008} for the adapted M-solutions to  BSVIEs of form \rf{bsvie-II}.
This important development has triggered extensive research on BSVIEs and their applications. For instance, Anh, Grecksch and Yong \cite{Anh-Grecksch-Yong 2011} investigated BSVIEs in Hilbert spaces; Shi, Wang and Yong \cite{Shi--Wang--Yong 2013} studied well-posedness of BSVIEs containing mean-fields (of the unknowns); Ren \cite{Ren 2010}, Wang and Zhang \cite{Wang-Zhang p} discussed BSVIEs with jumps; Overbeck and R\"oder \cite{Overbeck-Roder p} even developed a theory of path-dependent BSVIEs; Numerical aspect was considered by Bender and Pokalyuk \cite{Bender-Pokalyuk 2013}; relevant optimal control problems were studied by Shi, Wang and Yong \cite{Shi-Wang-Yong 2015},
Agram and {\O}ksendal \cite{Agram-Oksendal 2015}, Wang and Zhang \cite{Wang--Zhang 2017}, and Wang \cite{Wang 2018}; Wang and Yong \cite{Wang--Yong 2015} established various comparison theorems for both adapted solutions and adapted M-solutions to BSVIEs in multi-dimensional Euclidean spaces.

\ms
Recently, inspired by the Four-Step Scheme in the theory of forward-backward stochastic differential equations (FBSDEs, for short) (\cite{Ma--Yong 1999}),
in the Markovian frame:
\begin{align}
\label{SDE}&X(t)=x+\int_0^tb(s,X(s))ds+\int_0^t\si(s,X(s))dW(s),\\
\label{mbsvie}&Y(t)=\psi(t,X(T)) + \int_t^T g(t,s,X(s),Y(s),Z(t,s))ds -\int_t^T Z(t,s)dW(s),
\end{align}
Wang--Yong \cite{Wang--Yong 2018} proved that: If $\Theta(\cd,\cd,\cd)$ is a classical solution
to the following PDE:
\bel{FBPDE}\left\{\begin{aligned}
&\Th_s(t,s,x)+{1\over 2}\si(s,x)^\top\Th_{xx}(t,s,x)\si(s,x)+\Th_x(t,s,x)b(s,x)\\
&\qq\q +g(t,s,x,\Th(s,s,x),\Th_x(t,s,x)\si(s,x))=0,\q (t,s,x)\in[0,T]\times[t,T]\times\dbR^d,\\
&\Th(t,T,x)= \psi(t,x),\qq (t,x)\in[0,T]\times\dbR^d,
\end{aligned}\right.\ee
then
\bel{rep-sde-bsvie}
Y(t)=\Th(t,t,X(t)),\q Z(t,s)=\Th_x(t,s,X(s))\si(s,X(s)),~(t,s)\in[0,T]\times[t,T]
\ee
is the unique adapted solution to Markovian BSVIE \rf{mbsvie}, where
\begin{equation*}
\si(s,x)^\top\Th_{xx}(t,s,x)\si(s,x)=\sum_{i=1}^d\begin{pmatrix}\si_i(s,x)^\top\Th^1_{xx}(t,s,x)\si_i(s,x) \\
                                                                 \si_i(s,x)^\top\Th^2_{xx}(t,s,x)\si_i(s,x)\\
                                                                 \vdots\\
                                                                 \si_i(s,x)^\top\Th^m_{xx}(t,s,x)\si_i(s,x) \end{pmatrix},
\end{equation*}
with
\begin{equation*}
    \si(s,x)=\Big(\si_1(s,x),\si_2(s,x),\cdot\cdot\cdot,\si_d(s,x)\Big)\q \hbox{and}\q  \Th(t,s,x)=\begin{pmatrix}\Th^1(t,s,x) \\
                                                                                                     \Th^2(t,s,x)\\
                                                                                                        \vdots\\
                                                                                                     \Th^m(t,s,x)\end{pmatrix}.
\end{equation*}
They also proved that under some regularity and boundness conditions of the coefficients and the uniformly positive condition, i.e., there exists a constant $\bar\si>0$ such that
\bel{uni-pos}
|\si(s,x)\xi|^2\geq \bar\si|\xi|^2,\q\forall(s,x,\xi)\in[0,T]\times\dbR^d\times\dbR^d,
\ee
then system \rf{FBPDE} admits a unique classical solution.
This result provides a representation of adapted solutions via a solution to the (non-classical) partial differential equation \rf{FBPDE}, together with the solution $X(\cd)$
to the (forward) stochastic differential equation \rf{SDE}.
We emphasize that the above PDE is
non-local, because the $g$-term involves values
$\Th(s,s,x)$.
To our best knowledge, the PDEs of form \rf{FBPDE} appeared the first time in the study of time-inconsistent optimal control problems.
In the time-inconsistent optimal control problems, the PDE \rf{FBPDE} serves as an equilibrium HJB equation,
which is used to express the equilibrium strategy and equilibrium vale function (\cite{Yong 2012}, see also \cite{Wei--Yong--Yu 2017}, \cite{Mei--Yong 2017}).

\ms
In 1992, Pardoux--Peng \cite{Pardoux--Peng 1992} considered the following Markovian forward-backward stochastic differential equations (FBSDEs, for short):
\begin{align}
\label{SDET}&X^{t,x}(s)=x+\int_t^sb(r,X^{t,x}(r))dr+\int_t^s\si(r,X^{t,x}(r))dW(r),\\
\label{mbsde}&Y^{t,x}(s)=\psi(X^{t,x}(T)) + \int_s^T g(r,X^{t,x}(r),Y^{t,x}(r),Z^{t,x}(r))dr -\int_s^T Z^{t,x}(r)dW(r),
\end{align}
where $t,x\in[0,T)\times\dbR^d$ and $b(\cd),\si(\cd),\psi(\cd),g(\cd)$ are deterministic fuctions.
Apparently, under some mild conditions of the coefficients, the above FBSDE admits a unique adapted solution $(X^{t,x}(\cd),Y^{t,x}(\cd),Z^{t,x}(\cd))$.
In \cite{Pardoux--Peng 1992}, they obtained that: if the following PDEs
\bel{FBPDE-BSDE}\left\{\begin{aligned}
&\Th_s(s,x)+{1\over 2}\si(s,x)^\prime\Th_{xx}(s,x)\si(s,x)+\Th_x(s,x)b(s,x)\\
&\qq\q +g(s,x,\Th(s,x),\Th_x(s,x)\si(s,x))=0,\q (s,x)\in[0,T]\times\dbR^d,\\
&\Th(T,x)= \psi(x),\qq x\in\dbR^d
\end{aligned}\right.\ee
has a classical solution, then
\bel{rep-pde-sde-bsde}
Y^{t,x}(s)=\Th(s,X^{t,x}(s)),\q Z^{t,x}(s)=\Th_x(s,X^{t,x}(s))\si(s,X^{t,x}(s)),
\ee
which could be regarded as a special case of \rf{rep-sde-bsvie}.
But, more remarkable, under some regularity conditions (but without uniformly positive condition \rf{uni-pos}) of the coefficients,
they proved that
\bel{rep-sde-bsde-pde}\ti\Th(t,x)\deq Y^{t,x}(t),~(t,x)\in[0,T]\times\dbR^d\ee
is the unique classical solution to (possibly degenerate) parabolic PDE \rf{FBPDE-BSDE},
which is called the {\it nonlinear Feynman-Kac formula} and $Y^{t,x}(t);(t,x)\in[0,T]\times\dbR^d$ is usually called
a {\it probabilistic representation} of the solution to PDE \rf{FBPDE-BSDE}.
This result attracts  extensive research on the probabilistic representation of PDEs.
Among relevant works, we would like to mention Pardoux--Peng \cite{Pardoux--Peng 1994} for the doubly BSDEs and  stochastic PDEs;
Ekren, et al. \cite{Ekren 2014}, Peng--Wang\cite{Peng--Wang 2016}, Zhang \cite[Chapter 11]{Zhang 2017} for  the non-Markovian BSDEs and path-dependent PDEs.
Further, from a numerical application viewpoint, the BSDE representation leads to original probabilistic approximation scheme for the resolution in high
dimension of partial differential equations, as recently investigated in \cite{Pham 2015}.
It is then natural to ask:
Can we give a probabilistic representation of the solution to the following non-local PDEs \rf{FBTPDE}?
\bel{FBTPDE}\left\{\begin{aligned}
&\Th_s(t,s,x)+{1\over 2}\si(s,x)^\prime\Th_{xx}(t,s,x)\si(s,x)+\Th_x(t,s,x)b(s,x)\\
&\qq+g(t,s,x,\Th(t,s,x),\Th(s,s,x),\Th_x(t,s,x)\si(s,x))=0,\q (t,s,x)\in\D[0,T]\times\dbR^d,\\
&\Th(t,T,,x)= \psi(t,x),\qq (t,x)\in[0,T]\times\dbR^d.
\end{aligned}\right.\ee
The objective of this paper is to tackle this problem.
In the above PDEs, note that the nonlinear term $g(t,s,x,y,y',z)$ depends on both  $y$ and $y'$,
which also both appear in the equilibrium HJB equation of time-inconsistent optimal control problems.
Thus, we introduced  the following Markovian EBSVIE:
\begin{align}
\nn Y^{t,x}(s,r)&=\psi(s,X^{t,x}(T)) + \int_r^T g(s,\t,X^{t,x}(\t),Y^{t,x}(s,\t),Y^{t,x}(\t,\t),Z^{t,x}(s,\t))d\t \\
\label{mebsvie} &\qq\qq\qq-\int_r^T Z^{t,x}(s,\t)dW(\t),\q s\in [t,T),~ r\in[s,T],
\end{align}
where $X^{t,x}(\cd)$ is the unique strong solution to SDE \rf{SDET}.
To be more general, we first consider the EBSVIEs of form  \rf{ebsvie-I},
which is an extension of BSDEs \rf{ebsde} with parameters and BSVIEs \rf{bsvie}.
We shall establish the well-posedness of EBSVIEs \rf{ebsvie-I} using the  method introduced in Yong \cite{Yong 2008}.
Under an additional continuity condition, we get a $L^p(\Om;C([0,T];\dbR^m))$-norm estimate
of $Y(s,s);0\leq s\leq T$, where $L^p(\Om;C([0,T];\dbR^m))$ is defined in next section.
In the Markovian frame, by means of Malliavin calculus, we obtain some better regularity results for the adapted solutions to EBSVIEs \rf{FBTPDE}.
More precisely, we prove  that $Y^{s,x}(t,s)\in \mathcal{C}^{0,1,2}([0,T]\times[t,T]\times \dbR^d;\dbR^m)$, which follows that
\bel{rep-sde-bsvie-pde}\ti\Th(t,s,x)\deq Y^{s,x}(t,s)\ee
is the unique classical solution to the non-local PDE \rf{FBTPDE}.

\ms

The rest of this paper is organized as follows.
In Section \ref{Sec:Preliminaries}, we collect some preliminary results and introduce a few elementary notions.
Section \ref{Sec:Well-posedness} is devoted to the study of the well-posedness of EBSVIEs \rf{ebsvie-I}.
In section \ref{Sec:Regularity}, in the Markovian frame, we obtain some regularity property results for the adapted solutions to EBSVIEs \rf{FBTPDE}.
Finally, in section \ref{Sec:EBSVIE-PDE}, we give the  probabilistic representation of \rf{FBTPDE}.

\section{Preliminaries}\label{Sec:Preliminaries}
Recall that $\dbR^m$ is the usual $m$-dimensional Euclidean space and $\dbR^{m\times d}$ is the set of all $m\times d$ real matrices, endowed with the Frobenius inner product $\lan M,N\ran\mapsto\tr[M^\top N]$. We shall denote by $I_d$ the identity matrix of size $d$ and by $|M|$ the Frobenius norm of a matrix $M$.
For $0\les a<b\les T$, we denote by $\cB([a,b])$ the Borel $\si$-field on $[a,b]$
and define the following sets:
\begin{alignat*}{3}
\D[a,b] &\deq\big\{(t,s)\bigm|a\les t\les s\les b\big\},                  \q& \D^c[a,b] &\deq\big\{(t,s)\bigm|a\les s<t\les b\big\},\\
[a,b]^2 &\deq\big\{(t,s)\bigm|a\les t,s\les b\big\}=\D[a,b]\cup\D^c[a,b], \q& \D^*[a,b] &\deq\cl{\D^c[a,b]}.
\end{alignat*}
Note that $\D^*[a,b]$ is a little different from the complement $\D^c[a,b]$ of $\D[a,b]$ in $[a,b]^2$, since both $\D[a,b]$ and $\D^*[a,b]$ contain the diagonal line segment.
For any $t\in[0,T)$, let $\dbF^t=\{\cF^t_s\}_{s\geq 0}$ be the natural filtration of $W(s\vee t)-W(t);0\leq s\leq T$ augmented by all the $\dbP$-null sets in $\cF$.
By the definition of $\dbF=\{\cF_s\}_{s\geq 0}$, we know that $\dbF^0=\{\cF^0_s\}_{s\geq 0}=\{\cF_s\}_{s\geq 0}=\dbF$.
Let $\dbH$, $\dbH^\prime$  be two Euclidean spaces (which could be $\dbR^m$, $\dbR^{m\times d}$, etc.),
$p\in(1,\i]$ be a constant, and $k\geq 0$ be an integer.
We further introduce the following spaces of functions and processes:
\begin{align*}
C^k(\dbH;\dbH^\prime)
&\ts=\Big\{\f:\dbH\to\dbH^\prime~|~\hb{$\f(\cd)$ is $k$-th continuously differentable}\Big\},\\
C_b^k(\dbH;\dbH^\prime)
&\ts=\Big\{\f\in C^k(\dbH;\dbH^\prime)~|~\hb{for any $0<i\leq k$, the  $i$-th order partial derivatives}\\
&\ts\hp{=\Big\{\ }\hb{of $\f(\cd)$ are bounded}\Big\},\\
L^p(a,b;\dbH)&\ts=\Big\{\f:[a,b]\to\dbH~|~h(\cd)~\hb{is $\cB([a,b])$-measurable, }\int_a^b|h(s)|^pds<\i\Big\},\\
L^\i(a,b; L_{\cF_b}^p(\Om;\dbH))
&\ts=\Big\{\f:[a,b]\times\Om\to\dbH~|~\f(\cd)~\hb{is $\cB([a,b])\otimes\cF_b$--measurable,~}\\
&\ts\hp{=\Big\{\ }\ds\sup_{a\les s\les b}\dbE\big[|\f(s)|^p\big]<\i \Big\},\\
L^p_{\cF_b}(\Om;C^U([a,b];\dbH))
&\ts=\Big\{\f:[a,b]\times\Om\to\dbH~|~\f(\cd)~\hb{is $\cB([a,b])\otimes\cF_b$--measurable,~there exists a}\\
&\ts\hp{=\Big\{\ }  \hbox{ modulus of continuity $\rho:[0,\i)\to[0,\i)$} \hbox{ such that}~ \\
&\ts\hp{=\Big\{\ }|\f(t)-\f(s)|\les \rho(|t-s|),~(t,s)\in[a,b], \hbox{~and~} \dbE\big[\ds\sup_{a\les s\les b}|\f(s)|^p\big]<\i \Big\},\\
L_\dbF^p(\Om;C([a,b];\dbH))
&\ts=\Big\{\f:[a,b]\times\Om\to\dbH~|~\f(\cd)~\hb{is continuous, $\dbF$-adapted, }\\
&\ts\hp{=\Big\{\ } \dbE\big[\ds\sup_{a\les s\les b}|\f(s)|^p\big]<\i \Big\},\\
L_\dbF^p(\Om;C(\D[a,b];\dbH))
&\ts=\Big\{\f:\D[a,b]\times\Om\to\dbH~|\hb{for any $t\in[a,b)$, } \f(t,\cd)\in L_\dbF^p(\Om;C([t,b];\dbH)),\\
&\ts\hp{=\Big\{\ }\ds\sup_{a\les t\les b}\dbE\big[\ds\sup_{t\les s\les b}|\f(t,s)|^p\big]<\i \Big\},\\
L_\dbF^p(a,b;\dbH)
&\ts=\Big\{\f:[a,b]\times\Om\to\dbH~|~\f(\cd)~\hb{is $\dbF$-progressively measurable, }\\
&\ts\hp{=\Big\{\ }\dbE\int_a^b|\f(s)|^pds<\i \Big\},\\
L_\dbF^p(\D[a,b];\dbH)
&\ts=\Big\{\f:\1n\D[a,b]\1n\times\1n\Om\1n\to\1n\dbH~|\hb{for any}
~t\1n\in[a,b),~\f(t,\cd)~\hb{is $\dbF$-progressively }\\
&\ts\hp{=\Big\{\ }\qq \hb{measurable on $[t,b]$,}~\ds\sup_{a\les t\les b}\dbE\int_t^b|\f(t,s)|^pds<\i \Big\},\\
\cH^p[a,b]&=L_\dbF^p(\Om;C(\D[a,b];\dbR^m))\times L_\dbF^p(\D[a,b];\dbR^{m\times d}).
\end{align*}

For stochastic differential equation \rf{SDET}, we adopt the following assumption.
\begin{taggedassumption}{(F.1)}\label{ass:F.1}\rm
Let $b:[0,T]\times\dbR^d\to\dbR^d$ and $\si:[0,T]\times\dbR^d\to\dbR^{d\times d}$ be continuous functions.
For any $s\in[0,T]$, let $b(s,\cd)\in C^3_b(\dbR^d;\dbR^d)$ and $\si(s,\cd)\in C^3_b(\dbR^d;\dbR^{d\times d})$.
There exist three constants $C\geq 0$, $K_b\geq 0$,  $K_\si\geq 0$ such that:
\begin{align*}
&|b(s,0)|,~|\si(s,0)|\leq C,\q \forall s\in[0,T],\\
&|b_x(s,x)|, ~|b_{xx}(s,x)|,~ |b_{xxx}(s,x)|\leq K_b,\q\forall (s,x)\in[0,T]\times \dbR^d,\\
&|\si_x(s,x)|, ~|\si_{xx}(s,x)|,~ |\si_{xxx}(s,x)|\leq K_\si,\q \forall (s,x)\in[0,T]\times \dbR^d.
\end{align*}
\end{taggedassumption}

\begin{remark}\rm
Let {\rm\ref{ass:F.1}} hold. For any $(t,x)\in[0,T)\times\dbR^d$ and $p>1$, by the standard result \cite[ Chapter 1, Theorem 6.3]{Yong--Zhou 1999} of SDE,
the SDE \rf{SDET} admits a unique solution $X^{t,x}(\cd)\in L_\dbF^p(\Om;C([t,T];\dbR^d))$.
Moreover, we have $X^{t,x}(\cd)\in L_{\dbF^t}^p(\Om;C([t,T];\dbR^d))$.
\end{remark}

If $u(\cd)$ is a function of $x\in\dbR^d$, for any $h\in\dbR\setminus\{0\}$, let $\D^i_hu(x)\deq h^{-1}[u(x+he_i)-u(x)];~1\leq i\leq d$,
where $e_i$ denotes the $i$-th vector of an arbitrary orthonormal basis of $\dbR^d$.
Define $X^{t,x}_s=X^{t,x}_{s\vee t}$; $(t,s)\in[0,T]^2$, $x\in\dbR^d$.
The following result, whose proof is standard and can be found in \cite{Pardoux--Peng 1992}, establishes the regularity property of SDE \rf{SDET} under the assumption \autoref{ass:F.1}.
\begin{lemma}\label{lmm:well-posedness-SDE}
Let {\rm\ref{ass:F.1}} hold. For any $p\geq 2$, there exists a constant $C_p>0$ such that for any $t,t^\prime\in[0,T]$,
$x,x^\prime\in\dbR^d$, $i\in\{1,...,d\}$, $h,h^\prime\in\dbR\setminus \{0\}$,
\begin{align}
\label {lemma2.1.1}&\dbE\Big[\sup_{0\les s\les T}|X^{t,x}(s)|^p\Big]\les C_p(1+|x|^p),\q\dbE\Big[\sup_{0\les s\les T}|\D^i_h X^{t,x}(s)|^p\Big]\les C_p,\\
&\dbE\Big[\sup_{0\les s\les T}|X^{t,x}(s)-X^{t^\prime,x^\prime}(s)|^p\Big]\les C_p(1+|x|^p)(|x-x^\prime|^p+|t-t^\prime|^{{p\over 2}}),\\
&\dbE\Big[\sup_{0\les s\les T}|\D^i_h X^{t,x}(s)-\D^i_{h^\prime} X^{t^\prime,x^\prime}(s)|^p\Big]\les C_p(1+|x|^p)(|x-x^\prime|^p+|h-h^\prime|^p+|t-t^\prime|^{{p\over 2}}),\\
&\dbE\Big[\sup_{0\les s\les T}|\D^i_h X^{t,x}(s)-\D^i_{h^\prime} X^{t,x^\prime}(s)|^p\Big]\les C_p(|x-x^\prime|^p+|h-h^\prime|^p).
\end{align}
\end{lemma}
\autoref{lmm:well-posedness-SDE} immediately implies the following proposition (whose proof can be also found in \cite{Pardoux--Peng 1992}):
\begin{proposition}\label{diff-SDE}
For any $(t,s)\in\D[0,T]$ and $x\in\dbR^d$, the mapping $x\mapsto X^{t,x}(s)$ is a.s. differentiable.
The matrix of partial derivatives $\nabla X^{t,x}(s);t\leq s\leq T$ possesses a version which is a.s. continuous.
Further, $\nabla X^{t,x}(\cd)$ is the unique solution to the following linear stochastic differential equation:
\bel{rep-diff-SDE}
\nabla X^{t,x}(s)=I_d+\int_t^s b_x(r,X^{t,x}(r))\nabla X^{t,x}(r)dr+\int_t^s\si_x(r,X^{t,x}(r))\nabla X^{t,x}(r)dW(r).
\ee
\end{proposition}
Now, we briefly recall some relevant notations and results about Malliavin calculus, which will be used below.
Let $\Xi$ be the set of all (scalar) $\cF _T$-measurable random variables $\xi$ of form
\bel{xi}\xi=f\lt( \int_0 ^T h(s)dW(s)\rt),\ee
where $f\in C_b^1(\dbR^d;\dbR)$, $h(\cd)\equiv(h_1(\cd), . . . , h_d(\cd))$ with $h_i(\cd)\in L^\i(0,T;\dbR^d)$.
For any $\xi\in\Xi$, define
$$D^i_r\xi=\lt\langle f_x\lt( \int_0 ^T h(s)dW(s)\rt),h_i(r)\rt\rangle,\q 0\leq r\leq T,~1\leq i\leq d.$$
We call $D^i_r\xi;0\leq r\leq T$ the {\it Malliavin derivative} of $\xi$ with respect to $W^i(\cd)$.
Next, for any $\xi\in\Xi$ (of form \rf{xi}), we denote its $1,p$-norm by:
$$
\|\xi\|_{\dbD_{1,p}}^p\deq \dbE\Big[|\xi|^p+\int_0^T|D_r\xi|^pdr\Big].
$$
Clearly, $\|\cd\|_{\dbD_{1,p}}$ is a norm. Let $\dbD_{1,p}$ be the completion of $\Xi$ under the norm $\|\cd\|_{\dbD_{1,p}}$.
It is known \cite{Nualart 1995} that operator $D=(D^1,D^2,...,D^d)$ admits a closed extension on $\dbD_{1,p}$, and
$$ \xi~ \hbox{is $\cF_t$-measurable}\Rightarrow D_r^i\xi=0, \forall r\in(t,T], 1\leq i\leq d.$$
The following result is concerned with the Malliavan derivative of $X^{t,x}(\cd)$, we refer the reader to \cite[Lemma 1.1]{Pardoux--Peng 1992} for the proof.
\begin{lemma}\label{lmm:Malliavan-SDE}
Let {\rm\ref{ass:F.1}} hold. For any $(t,s)\in\D[0,T]$ and $x\in\dbR^d$, $X^{t,x}(s)\in(\dbD_{1,2})^d$, and a version of $\{D_rX^{t,x}(s);s,r\in[0,T]\}$ is given by:
\begin{enumerate}[(i)]
\item $D_rX^{t,x}(s)=0,~r\in[0,T]\setminus(t,s]$.

\item For any $t<r\leq T$, $\{D_rX^{t,x}(s);r\les s\les T\}$ is the unique solution of the linear SDE:
      \begin{align}
      \nn D_rX^{t,x}(s)&=\si(r,X^{t,x}(r))+\int_r^sb_x(\t,X^{t,x}(\t))D_rX^{t,x}(\t)d\t\\
      \label{mallivan-SDE}  &\hp{=} +\sum_{i=1}^d\int_r^s\si^i_x(\t,X^{t,x}(\t))D_rX^{t,x}(\t)dW^i(\t),
      \end{align}
      where  $\si^i$ denotes the $i$-th column of the matrix $\si$.
\end{enumerate}

\end{lemma}
\begin{remark}\label{remark-sde}\rm
By the uniqueness of the solution to SDE \rf{mallivan-SDE}, combining \rf{rep-diff-SDE} and \rf{mallivan-SDE}, we have
\bel{diff-malli-sde}
 D_rX^{t,x}(s)=\nabla X^{t,x}(s)(\nabla X^{t,x}(r))^{-1}\si(r,X^{t,x}(r)),\q 0\leq t\leq r\leq s\leq T.
\ee
\end{remark}
Now, for any $t\in[0,T)$, let us consider the following BSDE:
\bel{BSDE}
 Y(s)=\xi + \int_s^T g(r,Y(r),Z(r))dr-\int_s^T Z(r)dW(r),\q s\in[t,T].
\ee
We first introduce the following hypothesis.
\begin{taggedassumption}{(B.0)}\label{B0}
\rm For any $p\geq 2$ and $t\in[0,T)$, let the generator $g:[t,T]\times \dbR^m\times\dbR^{m\times d}\times\Om\to\dbR^m$ be $\cB([t,T]\times \dbR^m\times\dbR^{m\times d})\otimes\cF_T$-measurable
such that $s\mapsto g(s,y,z)$ is $\dbF$-progressively measurable for all $(y,z)\in \dbR^m\times\dbR^{m\times d}$ and
\bel{B0-f(s,0,0)}
\dbE\lt(\int_t^T|g(s,0,0)|ds\rt)^p<\i.
\ee
Moreover, there is a constant $L>0$ such that
\bel{|f-f|}
|g(s,y_1,z_1)-g(s,y_2,z_2)|\les L\big[|y_1-y_2|+|z_1-z_2|\big],~(s,y_i,z_i)\in[t,T]\times\dbR^m\times\dbR^{m\times d},~i=1,2.
\ee
\end{taggedassumption}
Under \autoref{B0}, we have the following result.
\begin{lemma}\label{lmm:BSDE}
Let $t\in[0,T)$ and {\rm \ref{B0}} hold. Then for any $\xi\in L_{\cF_T}^p(\Om;\dbR^m)$,
BSDE \rf{BSDE} admits a unique adapted solution
$(Y(\cd),Z(\cd))\in L_{\dbF}^p(\Om;C([t,T];\dbR^m))\times L^p_{\dbF}(t,T;\dbR^{m\times d})$ and there is a constant $C_p>0$ such that:
\bel{BSDE-ESTIMATE}
\dbE\Big[\sup_{s\in[t,T]}|Y(s)|^p\Big]+\dbE\Big[\int_t^T|Z(s)|^2 ds\Big]^{p\over 2}\leq  C_p\dbE |\xi|^p+C_p\dbE\Big[\int_t^T|g(s,0,0)|ds\Big]^p.
\ee
In addition, if
\bel{lemma2.3-xi}
\dbE_t[|\xi|^p]<\i,
\ee
and
\bel{lemma2.3-f(s,0,0)1,2}
g(s,y,z)=g_1(s,y,z)+g_2(s,y,z),~(s,y,z)\in[0,T]\times\dbR^m\times\dbR^{m\times d},
\ee
where $g_1(\cd)$ is a deterministic function and $g_2(\cd)$ is a stochastic process satisfying
\bel{lemma2.3-f(s,0,0)2}
\int_t^T|g_1(s,0,0)|ds<\i,\q\dbE_t\Big[\int_t^T|g_2(s,0,0)|^pds\Big]<\i,
\ee
we have
\begin{align}
\nn&\sup_{s\in[t,T]}\dbE_t\Big[|Y(s)|^p\Big]+\dbE_t\Big[\int_t^T|Z(s)|^2 ds\Big]^{p\over 2}\\
\label{BSDE-ESTIMATE2}&\q\leq  C_p\dbE_t\Big[|\xi|^p+\int_t^T|g_2(s,0,0)|^p ds\Big]+C_p\Big[\int_t^T|g_1(s,0,0)|ds\Big]^p.
\end{align}
\end{lemma}
\begin{proof}
The existence and uniqueness of the adapted solution to \rf{BSDE} and the estimate \rf{BSDE-ESTIMATE}
can be found in \cite{Yong--Zhou 1999}. Applying It\^{o}'s formula to $s\mapsto |Y(s)|^p$, we have
\begin{align}
\nn&\dbE_t[|Y(s)|^p]+{p(p-1)\over2}\dbE_t\Big[\int^T_s|Y(r)|^{p-2}|Z(r)|^2dr\Big]\\
\label{lamma-bsde-ito}&\q\leq \dbE_t[|\xi|^p]+p\dbE_t\Big[\int^T_s|Y(r)|^{p-1}|g(r,Y(r),Z(r))|dr\Big].
\end{align}
Under \rf{lemma2.3-xi}--\rf{lemma2.3-f(s,0,0)1,2}--\rf{lemma2.3-f(s,0,0)2}, by Young inequality and H\"older inequality, we have
\begin{align}
\nn&\sup_{s\in[t,T]}\dbE_t[|Y(s)|^p]+{p(p-1)\over2}\dbE_t\Big[\int^T_s|Y(r)|^{p-2}|Z(r)|^2dr\Big]\\
\nn&~\leq \dbE_t[|\xi|^p]+C_p\dbE_t\Big[\int^T_t\big(|Y(r)|^{p-1}|g_1(r,0,0)|+|Y(r)|^{p-1}|g_2(r,0,0)|\\
\nn&\qq\qq\qq\qq\qq+|Y(r)|^p+|Y(r)|^{p-1}|Z(r)|\big)dr\Big] \\
\nn&~\leq \dbE_t[|\xi|^p]+{p(p-1)\over4}\dbE_t\Big[\int^T_t|Y(r)|^{p-2}|Z(r)|^2dr\Big]+C_p\dbE_t\Big[\int_t^T|Y(r)|^pdr\Big]
\\
\label{lamma-bsde-gornwall}&\qq+C_p\dbE_t\Big[\int_t^T|g_2(r,0,0)|^pdr\Big]+{1\over2}\sup_{s\in[t,T]}\dbE_t[|Y(s)|^p]+C_p\Big[\int_t^T|g_1(r,0,0)|dr\Big]^p.
\end{align}
By \rf{lamma-bsde-gornwall} and Gr\"{o}nwall's inequality, we have the estimate \rf{BSDE-ESTIMATE2}.
\end{proof}

\section{Well-posedness}\label{Sec:Well-posedness}
In this section, we will establish well-posedness of the adapted solution to EBSVIE \rf{ebsvie-I}.
We first adopt the following assumption [which is comparable with  {\rm \ref{B0}}]:
\begin{taggedassumption}{(B.1)}\label{ass:B1}\rm
Let the generator $g:\D[0,T]\times \dbR^m\times\dbR^m\times\dbR^{m\times d}\times\Om\to\dbR^m$ be $\cB(\D[0,T]\times \dbR^m\times\dbR^m\times\dbR^{m\times d})\bigotimes\cF_T$-measurable
such that $s\mapsto g(t,s,y,y^\prime,z)$ is $\dbF$-progressively measurable for all $(t,y,y^\prime,z)\in [0,T)\times\dbR^m\times\dbR^m\times\dbR^{m\times d}$ and
\bel{B1-g(t,s)}
\sup_{t\in[0,T]}\dbE\lt(\int_t^T|g(t,s,0,0,0)|ds\rt)^p<\i,
\ee
where $p\geq 2$ is a constant.
Moreover, there is a constant $L>0$ such that
\begin{align}
\nn|g(t,s,y_1,y'_1,z_1)-g(t,s,y_2, y'_2, z_2)|\leq L(|y_1- y_2|+|y'_1- y'_2|+|z_1- z_2|),\\
\label{B1-L}\forall (t,s)\in\D[0,T],~(y_1,y'_1,z_1),~( y_2, y'_2, z_2)\in\dbR^m\times\dbR^m\times\dbR^{m\times d}.
\end{align}
\end{taggedassumption}
We now give the following well-posedness result for EBSVIE \rf{ebsvie-I}.

\begin{theorem}\label{thm-wellposedness}
Let $p\geq 2$ and {\rm \ref{ass:B1}} hold. Then for any $\psi(\cd)\in L^\i(0,T; L_{\cF_T}^p(\Om;\dbR^m))$, EBSVIE \rf{ebsvie-I} admits a unique adapted solution $(Y(\cd,\cd),Z(\cd,\cd))\in\cH^p[0,T]$,
and the following estimate holds:
\begin{align}
\nn& \sup_{t\in[0,T]}\dbE\Big[\sup_{s\in[t,T]}|Y(t,s)|^p\Big]+\sup_{t\in[0,T]}\dbE\Big[\int_t^T|Z(t,s)|^2 ds\Big]^{p\over 2}\\
\label{EBSVIE-estimate}&\q \leq  C_p\sup_{t\in[0,T]}\dbE |\psi(t)|^p+C_p\sup_{t\in[0,T]}\dbE\Big[\int_t^T|g(t,s,0,0,0)|ds\Big]^p.
\end{align}
For $i=1,2$, let $g_i(\cd)$ satisfy {\rm \ref{ass:B1}}, $\psi_i(\cd)\in L^\i(0,T; L_{\cF_T}^p(\Om;\dbR^m))$,
and $(Y_i(\cd,\cd),Z_i(\cd,\cd))\in\cH^p[0,T]$ be the unique adapted solution of EBSVIE \rf{ebsvie-I} corresponding to $g_i(\cd)$, $\psi_i(\cd)$, respectively, then
\begin{align}
\nn& \sup_{t\in[0,T]}\dbE\Big[\sup_{s\in[t,T]}|Y_1(t,s)-Y_2(t,s)|^p\Big]+\sup_{t\in[0,T]}\dbE\Big[\int_t^T|Z_1(t,s)-Z_2(t,s)|^2 ds\Big]^{p\over 2}\\
\nn& \q\leq  C_p\bigg\{\sup_{t\in[0,T]}\dbE\Big[\int_t^T\big|g_1(t,s,Y_1(t,s),Y_1(s,s),Z_1(t,s))-g_2(t,s,Y_1(t,s),Y_1(s,s),Z_1(t,s))\big|ds\Big]^p \\
\label{EBSVIE-sta-estimate}&\qq\qq+\sup_{t\in[0,T]}\dbE|\psi_1(t)-\psi_2(t)|^p\bigg\}.
\end{align}
\end{theorem}
\begin{proof}
We first prove that EBSVIE \rf{ebsvie-I} admits a unique adapted solution. The proof will be divided into three steps.
\ms

{\it Step 1:  Existence and uniqueness of the adapted solution to EBSVIE \rf{ebsvie-I} on $[S,T]$ for some $S\in[0,T)$.}

\ms
For any $(y(\cd,\cd),z(\cd,\cd))\in\cH^p[S,T]$, consider the following EBSVIE:
\bel{EBSVIE-BSDE-Parameter}
Y(t,s)=\psi(t) + \int_s^T g(t,r,Y(t,r),y(r,r),Z(t,r))dr - \int_s^T Z(t,r)dW(r).
\ee
The above EBSVIE can be regarded as a family of BSDEs parameterized by $t\in[S,T]$.
For any $t\in[S,T]$, by \autoref{lmm:BSDE}, the above EBSVIE (or BSDE) admits a unique adapted solution $(Y(t,\cd),Z(t,\cd))\in L_{\dbF}^p(\Om;C(t,T;\dbR^m))\times L^p_{\dbF}(t,T;\dbR^{m\times d})$ satisfying
\begin{equation*}
\dbE\Big[\sup_{s\in[t,T]}|Y(t,s)|^p\Big]+\dbE\Big[\int_t^T|Z(t,s)|^2 ds\Big]^{p\over 2}\leq  C_p\dbE |\psi(t)|^p+C_p\dbE\Big[\int_t^T|g(t,s,0,y(s,s),0)|ds\Big]^p,\\
\end{equation*}
which implies
\begin{align}
\nn& \sup_{t\in[S,T]}\dbE\Big[\sup_{s\in[t,T]}|Y(t,s)|^p\Big]+\sup_{t\in[S,T]}\dbE\Big[\int_t^T|Z(t,s)|^2 ds\Big]^{p\over 2}\\
\label{well-map}&\q \leq  C_p\sup_{t\in[S,T]}\dbE |\psi(t)|^p+C_p\sup_{t\in[S,T]}\dbE\Big[\int_t^T|g(t,s,0,y(s,s),0)|ds\Big]^p<\i.
\end{align}
Thus, we can define a map $\Th:\cH^p[S,T]\to\cH^p[S,T]$ by
\bel{Th}\Th(y(\cd,\cd),z(\cd,\cd))=(Y(\cd,\cd),Z(\cd,\cd)),\q (y(\cd,\cd),z(\cd,\cd))\in\cH^p[S,T].\ee
We claim that the map $\Th(\cd,\cd)$ is a contraction when $T-S>0$ is small.
To prove this, let $(\bar y(\cd,\cd),\bar z(\cd,\cd))\in\cH^p[S,T]$, and $(\bar Y(\cd,\cd),\bar Z(\cd,\cd))= \Th(\bar y(\cd,\cd),\bar z(\cd,\cd))$.
By the estimate \rf{BSDE-ESTIMATE} in \autoref{lmm:BSDE} and H\"{o}lder's inequality, we have
\begin{align}
\nn &\sup_{t\in[S,T]}\dbE\Big[\sup_{s\in[t,T]}|Y(t,s)-\bar Y(t,s)|^p\Big]+\sup_{t\in[S,T]}\dbE\Big[\int_t^T|Z(t,s)-\bar Z(t,s)|^2ds\Big]^{p\over 2}\\
\nn &\q \leq C_p\sup_{t\in[S,T]}\dbE \Big[\int_t^T|g(t,s,Y(t,s),y(s,s),Z(t,s))-g(t,s,Y(t,s),\bar y(s,s),Z(t,s))|ds\Big]^p\\
\nn &\q \leq C_p(T-S)^{p}\sup_{t\in[S,T]}\dbE|y(t,t)-\bar y(t,t)|^p\\
\label{CONTRAT} &\q \leq C_p(T-S)^{p}\sup_{t\in[S,T]}\dbE\lt[\sup_{t\leq s\leq T}|y(t,s)-\bar y(t,s)|^p\rt].
\end{align}
By \rf{CONTRAT}, when $T-S>0$ is small enough, the map $\Th(\cd,\cd)$ is a contraction on the set $\cH^p[S,T]$.
Hence, EBSVIE \rf{ebsvie-I} admits a unique adapted solution on  $[S,T]$.
Note that the choice of $T-S$ is independent of $\psi(\cd)$.

\ms
{\it Step 2: A family of BSDEs is solvable on $[S,T]$.}
\ms

We have seen that the value $(Y(t,s),Z(t,s));S\leq t\leq s\leq T$ is already determined, the region marked $\textcircled{\small 1}$ in the following figure.
Note that for any $t\in[S,T]$, $Y(t,s);t\leq s\leq T$ is continuous.
Thus, $Y(t,t);S\leq t\leq T$ is well-defined, the red line between the region marked $\textcircled{\small 1}$ and $\textcircled{\small 3}$.

\ms

\vskip-1cm

\setlength{\unitlength}{.01in}
~~~~~~~~~~~~~~~~~~~~~~~~~~~~~~~~~~~~~~~~~~~~~~~\begin{picture}(230,240)
\put(0,0){\vector(1,0){170}}
\put(0,0){\vector(0,1){170}}
\put(110,0){\line(0,1){150}}
\put(150,0){\line(0,1){150}}
\put(0,110){\line(1,0){150}}
\put(0,150){\line(1,0){150}}
\thicklines
\put(0,0){\color{red}\line(1,1){150}}
\put(122,137){\makebox(0,0){$\textcircled{\small 1}$}}
\put(55,130){\makebox(0,0){$\textcircled{\small 2}$}}
\put(135,120){\makebox(0,0){$\textcircled{\small 3}$}}
\put(130,55){\makebox(0,0){$\textcircled{\small 4}$}}
\put(-10,150){\makebox(0,0)[b]{$\scriptstyle T$}}
\put(150,-12){\makebox(0,0)[b]{$\scriptstyle T$}}
\put(-15,105){\makebox(0,0)[b]{$\scriptstyle S$}}
\put(105,-12){\makebox(0,0)[b]{$\scriptstyle S$}}
\put(180,-5){\makebox{$t$}}
\put(0,180){\makebox{$s$}}
\put(35,80){\makebox(0,0){$\scriptstyle\D[0,S]$}}
\put(75,35){\makebox(0,0){$\scriptstyle\D^*[0,S]$}}
\end{picture}
\bs

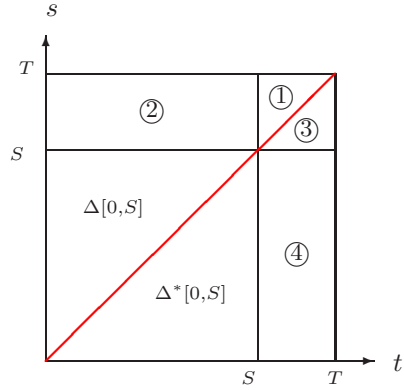
\centerline{(Figure 1)}

\bs

Hence, the following can be defined:
$$g^S(t,s,y,z)=g(t,s,y,Y(s,s),z),\q (t,s,y,z)\in[0,S]\times[S,T]\times\dbR^m\times\dbR^{m\times d}.$$
Consider the following BSDEs parameterized by $t\in[0,S]$:
\bel{SFIE-BSDE-Parameter}
Y(t,s)=\psi(t) + \int_s^T g^S(t,r,Y(t,r),Z(t,r))dr - \int_s^T Z(t,r)dW(r),\q s\in[S,T].
\ee
For all $t\in[0,S]$, by \autoref{lmm:BSDE}, the above BSDE admits a unique solution $(Y(t,s),Z(t,s));s\in[S,T]$,
and by the definition of $g^S(t,r,y,z)$, we see that $(Y(\cd,\cd),Z(\cd,\cd))$ satisfies
\bel{SFIE-BSVIE-BSDE}
Y(t,s)=\psi(t) + \int_s^T g(t,r,Y(t,r),Y(r,r),Z(t,r))dr - \int_s^T Z(t,r)dW(r),\q s\in[S,T].
\ee
Thus, this step uniquely determines the values $(Y(t,s),Z(t,s))$ for $(t,s)\in[0,S]\times[S,T]$, the region marked $\textcircled{\small 2}$ in the above figure.

\ms
{\it Step 3: Complete the proof by induction. }
\ms

By step 1--2, we have uniquely determined
\bel{y-z}\left\{\begin{aligned}
   \qq Y(t,t),\qq &\qq t\in[S,T],\\
   Y(t,s),Z(t,s),&\qq (t,s)\in\D[S,T]\bigcup[0,S]\times[S,T].
\end{aligned}\right.\ee
Now, we consider the following  EBSVIE on $[0,S]$:
\bel{SFIE-BSVIE}
Y(t,s)=Y(t,S) + \int_s^S g(t,r,Y(t,r),Y(r,r),Z(t,r))dr - \int_s^S Z(t,r)dW(r).
\ee
We see that the choice of $T-S$ is independent of $\psi(\cd)$. Hence the above procedure can be repeated.
Then we can use induction to finish the proof of the existence and uniqueness of adapted solution to EBSVIE \rf{ebsvie-I}.

\ms

We next prove the estimate \rf{EBSVIE-estimate}.
For the unique adapted solution $(Y(\cd,\cd),Z(\cd,\cd))\in\cH^p[0,T]$ to EBSVIE \rf{ebsvie-I}, consider the following BSVIE:
\bel{BSDE-Parameter}
\eta(t,s)=\psi(t) + \int_s^T g(t,r,\eta(t,r),Y(r,r),\zeta(t,r))dr - \int_s^T \zeta(t,r)dW(r),~(t,s)\in\D[0,T],
\ee
which is actually a family of BSDEs with parameter $t$.
For any $t\in[0,T]$, by \autoref{lmm:BSDE}, the above BSDE admits a unique solution $(\eta(t,\cd),\z(t,\cd))\in L_{\dbF}^p(\Om;C(t,T;\dbR^m))\times L^p_{\dbF}(t,T;\dbR^{m\times d})$.
By the above steps, we have
\bel{EBSVIE-BSDE}
Y(t,s)=\eta(t,s),~ Z(t,s)=\z(t,s),\q (t,s)\in\D[0,T].
\ee
By \autoref{lmm:BSDE}, there is a generic constant $C_p>0$ (which could be different from line to line) such that:
\begin{align}
\nn& \dbE\Big[\sup_{s\in[t,T]}|Y(t,s)|^p\Big]+\dbE\Big[\int_t^T|Z(t,s)|^2 ds\Big]^{p\over 2}\\
\label{ebsvie-bsde-estimate}&\q \leq  C_p\dbE |\psi(t)|^p+C_p\dbE\Big[\int_t^T|g(t,s,Y(s,s),0,0)|ds\Big]^p\\
\nn&\q \leq  C_p\dbE |\psi(t)|^p+C_p\dbE\Big[\int_t^T|g(t,s,0,0,0)|ds\Big]^p+C_p \dbE\int_t^T|Y(s,s)|^pds.
\end{align}
It follows that
\begin{align}
\nn& \dbE\Big[|Y(t,t)|^p\Big]+\dbE\Big[\int_t^T|Z(t,s)|^2 ds\Big]^{p\over 2}\\
\label{ebsvie-bsde-estimate1}&\q \leq  C_p\dbE |\psi(t)|^p+C_p\dbE\Big[\int_t^T|g(t,s,0,0,0)|ds\Big]^p+C_p \dbE\int_t^T|Y(s,s)|^pds.
\end{align}
By \rf{ebsvie-bsde-estimate1} and Gr\"{o}nwall's  inequality, we obtain
\begin{align}
\nn& \sup_{t\in[0,T]}\dbE\Big[|Y(t,t)|^p\Big]+\sup_{t\in[0,T]}\dbE\Big[\int_t^T|Z(t,s)|^2 ds\Big]^{p\over 2}\\
\label{ebsvie-bsde-estimate2}&\q \leq  C_p\sup_{t\in[0,T]}\dbE |\psi(t)|^p+C_p\sup_{t\in[0,T]}\dbE\Big[\int_t^T|g(t,s,0,0,0)|ds\Big]^p.
\end{align}
Combining this with \rf{ebsvie-bsde-estimate}, we have the estimate \rf{EBSVIE-estimate}.
Similarly, we obtain the stability estimate \rf{EBSVIE-sta-estimate}.
\end{proof}

We now would like to look some better regularity for the adapted solution of EBSVIEs under additional
conditions.  More precisely, we introduce the following assumption [comparing with \ref{ass:B1}].

\begin{taggedassumption}{(B2)}\label{B2}\rm
Let $g_1:[0,T]^2\times\dbR^m\times\dbR^m\times\dbR^{m\times d}\to\dbR^m$ be $\cB([0,T]^2\times \dbR^m\times\dbR^m\times\dbR^{m\times d})$-measurable such that
$$
\sup_{t\in[0,T]}\int_t^T|g_1(t,s,0,0,0)|ds<\i.
$$
Let $g_2:[0,T]^2\times\dbR^m\times\dbR^m\times\dbR^{m\times d}\times\Om\to\dbR^m$ be $\cB([0,T]^2\times \dbR^m\times\dbR^m\times\dbR^{m\times d})\bigotimes\cF_T$-measurable such that for every $(t,y,y',z)\in[0,T]\times\dbR^m\times\dbR^m\times\dbR^{m\times d}$, $s\mapsto g_2(t,s,y,y',z)$ is $\dbF$-progressively measurable and
$$
\dbE\bigg[\sup_{t\in[0,T]}\int_t^T|g_2(t,s,0,0,0)|^pds\bigg]<\i.
$$
Let $$g(\cd)=g_1(\cd)+g_2(\cd).$$
There exists a modulus of continuity $\rho:[0,\i)\to[0,\i)$ (a continuous and monotone increasing function with $\rho(0)=0$) such that
\begin{align*}
&|g(t_1,s,y,y',z)-g(t_2,s,y,y',z)|\les\rho(|t_1-t_2|),\\
&\qq\qq\qq \forall~t_1,t_2,s\in[0,T],~(y,y',z)\in\dbR^m\times\dbR^m\times\dbR^{m\times d}.
\end{align*}
\end{taggedassumption}

Note that in \ref{B2}, the generator $g(t,s,y,y',z)$ is defined for $(t,s)$ in the square domain $[0,T]^2$ instead of the triangle domain $\D[0,T]$.

\begin{theorem}\label{thm-y-continuity}
Let $p\geq 2$ and {\rm \ref{ass:B1}--\ref{B2}} hold.
For any $\psi(\cd)\in L^p_{\cF_T}(\Om;C^U(0,T;\dbR^m))$, let $(Y(\cd,\cd),Z(\cd,\cd))\in\cH^p[0,T]$ be the unique adapted solution to EBSVIE \rf{ebsvie-I},
then $Y(t,t);0\leq t\leq T$ is continuous and the following estimate holds:
\begin{align}
\nn\dbE\Big[\sup_{t\in[0,T]}|Y(t,t)|^p\Big]&\leq C_p\dbE\Big[\sup_{t\in[0,T]}|\psi(t)|^p+\sup_{t\in[0,T]}\int_t^T|g_2(t,s,0,0,0)|^p ds\Big]\\
\label{pro-EBSVIE-supestimate}&\q+C_p\sup_{t\in[0,T]}\Big[\int_t^T|g_1(t,s,0,0,0)| ds\Big]^p.
\end{align}
\end{theorem}
\begin{proof}
 Without loss of generality, let us assume that
$$|\psi(t')-\psi(t)|\les\rho(|t-t^\prime|),\q \forall~t,t'\in[0,T],$$
with the same modulus of continuity $\rho(\cd)$ given in \ref{B2}.

\ms
For any $t,t'\in[0,T]$, let
\bel{def-g0}
 g^0(t,t',s)=g(t,r,Y(t,r),Y(r,r),Z(t,r))-g(t',r,Y(t,r),Y(r,r),Z(t,r)).
\ee
Then we have
\begin{align}
\nn Y(t,s)-Y(t',s)&=\psi(t)-\psi(t') - \int_s^T \big[Z(t,r)-Z(t',r)\big] dW(r)\\
 \nn             &\q+\int_s^T\big[g(t,r,Y(t,r),Y(r,r),Z(t,r))-g(t',r,Y(t',r),Y(r,r),Z(t',r))\big]dr\\
\nn               &=\psi(t)-\psi(t') - \int_s^T \big[Z(t,r)-Z(t',r)\big] dW(r)\\
 \nn             &\q+\int_s^T  \big[g(t',r,Y(t,r),Y(r,r),Z(t,r))-g(t',r,Y(t',r),Y(r,r),Z(t',r))\big]dr\\
 \nn               &\q+\int_s^Tg^0(t,t',r)dr\\
 \nn                 &=\psi(t)-\psi(t') - \int_s^T \big[Z(t,r)-Z(t',r)\big] dW(r)\\
 \nn             &\q+\int_s^T  \big[g_y(t',r) (Y(t,r)-Y(t',r))+g_z(t',r) (Z(t,r)-Z(t',r))\big]dr\\
 \label{Y(t,s)-Y(t',s)}              &\q+\int_s^Tg^0(t,t',r)dr,
\end{align}
where
\begin{align}
\nn g_y(t',r)&=\int_0^1 g_y(t',r,\lambda Y(t,r)+(1-\lambda)Y(t',r),Y(r,r),\lambda Z(t,r)+(1-\lambda)Z(t',r))d\lambda,\\
\nn g_z(t',r)&=\int_0^1 g_z(t',r,\lambda Y(t,r)+(1-\lambda)Y(t',r),Y(r,r),\lambda Z(t,r)+(1-\lambda)Z(t',r))d\lambda.
\end{align}
Thus, the above EBSVIE \rf{Y(t,s)-Y(t',s)} can be regarded as a BSDE on $[0,T]$.
Note that $\psi(t)-\psi(t')$ and $g^0(t,t',r)$ are both uniformly bounded.
For any $s\in[0,T)$, by \rf{BSDE-ESTIMATE2} in \autoref{lmm:BSDE} , we have
\begin{align}
\nn\sup_{r\in[s,T]}\dbE_s\Big[|Y(t,r)-Y(t',r)|^p\Big]&\leq  C_p\dbE_s |\psi(t)-\psi(t')|^p+C_p\dbE_s\Big[\int_s^T|g^0(t,t',r)|^pdr\Big]\\
\label{y-continuous-t}                            &\leq  C_p\big[\rho(|t-t'|)\big]^p.
\end{align}
Let $r=s$, we have
$$|Y(t,s)-Y(t',s)| \leq C_p\rho(|t-t'|),\q s\in[0,T],$$
which leads to
$$\lim_{|t-t'|\to0}\[\sup_{s\in[0,T]}|Y(t,s)-Y(t',s)|\]=0,~\as $$
On the other hand, since $Y(t,\cd)\in L_\dbF^2(\Om;C(0,T;\dbR^m))$ for any $t\in[0,T]$, one has
$$\lim_{|s-s'|\to 0}|Y(t,s)-Y(t,s')|=0,\q\forall t\in[0,T],~\as$$
It follows that $(t,s)\mapsto Y(t,s)$ is continuous, i.e.,
$$\lim_{(t',s')\to(t,s)}|Y(t',s')-Y(t,s)|=0,\q\forall(t,s)\in[0,T]^2,~\as$$
Consequently, $t\mapsto Y(t,t)$ is continuous.

\ms
Next, we prove the estimate \rf{pro-EBSVIE-supestimate}.
For $i=1,2$ and any $n\geq 0$, define
\bel{define-g0n}g_i^n(t,s,0,0,0)\deq\left\{\begin{aligned}
   g_i(t,s,0,0,0),\q &\q \hbox{if}~ |g_i(t,s,0,0,0)|\leq n,\\
           \qq 0,\qq   \q&\q \hbox{if} ~|g_i(t,s,0,0,0)|> n,
\end{aligned}\right.\ee
\begin{align}
\nn g_i^n(t,s,y,y',z)&\deq g_i(t,s,y,y',z)-g_i(t,s,0,0,0)+g_i^n(t,s,0,0,0),\\
\nn g^n(t,s,y,y',z)&\deq g^n_1(t,s,y,y',z)+g^n_2(t,s,y,y',z), \\
\label{define-gn}&\qq(t,s,y,y',z)\in[0,T]^2\times\dbR^m\times\dbR^m\times\dbR^{m\times d},
\end{align}
and
\bel{define-psin}\psi^n(t)\deq\left\{\begin{aligned}
  \psi(t),\q &\q \hbox{if}~ |\psi(t)|\leq n,\\
           \q0,\q   \q&\q \hbox{if} ~|\psi(t)|> n.
\end{aligned}\right.\ee
Note that $g^n(\cd)$ satisfies the assumptions {\rm \ref{ass:B1}} and
\bel{estimate-gn}
|g^n(t,s,0,0,0)|\leq 2n, \q|\psi^n(t)|\leq n,\q (t,s)\in [0,T]^2.
\ee
By \autoref{thm-wellposedness}, the following EBSVIE
\bel{ebsvien} Y^n(t,s)=\psi^n(t) + \int_s^T g^n(t,r,Y^n(t,r),Y^n(r,r),Z^n(t,r))dr - \int_s^T Z^n(t,r)dW(r) \ee
admits a unique adapted solution $(Y^n(\cd),Z^n(\cd))\in\cH^p[0,T]$.
Further, similar to the proof of \rf{BSDE-ESTIMATE2} in \autoref{lmm:BSDE} and \rf{EBSVIE-estimate} in \autoref{thm-wellposedness}, we have
\begin{align}
\nn \sup_{s\in[t,T]}\sup_{r\in[s,T]}\dbE_t\Big[|Y^n(s,r)|^p\Big]&\leq  C_p\sup_{s\in[t,T]}\dbE_t |\psi^n(s)|^p+C_p\sup_{s\in[t,T]}\dbE_t\Big[\int_s^T|g_2^n(s,r,0,0,0)|^p dr\Big]\\
\nn &\q +C_p\sup_{s\in[t,T]}\Big[\int_s^T|g_1^n(s,r,0,0,0)| dr\Big]^p\\
\label{EBSVIEn-estimate}&<\i.
\end{align}
Let $r=s$, $s=t$, we have
\begin{align}
\nn|Y^n(t,t)|^p &\leq  C_p\sup_{s\in[t,T]}\dbE_t |\psi^n(s)|^p+C_p\sup_{s\in[t,T]}\dbE_t\Big[\int_s^T|g_2^n(s,r,0,0,0)|^p dr\Big]\\
\nn &\q +C_p\sup_{s\in[t,T]}\Big[\int_s^T|g_1^n(s,r,0,0,0)| dr\Big]^p\\
\nn&\leq  C_p\dbE_t \Big[\sup_{s\in[t,T]}|\psi^n(s)|^p\Big]+C_p\dbE_t\Big[\sup_{s\in[t,T]}\int_s^T|g_2^n(s,r,0,0,0)|dr\Big]^p\\
\label{EBSVIEnt-Et-estimate} &\q +C_p\sup_{s\in[t,T]}\Big[\int_s^T|g_1^n(s,r,0,0,0)| dr\Big]^p.
\end{align}
By \rf{EBSVIEnt-Et-estimate}, Doob's maximal inequality, and \rf{define-g0n}--\rf{define-gn}--\rf{define-psin}, we have
\begin{align}
\nn\dbE\Big[\sup_{t\in[0,T]}|Y^n(t,t)|^p\Big] &\leq  C_p\dbE\Big\{\sup_{t\in[0,T]}\dbE_t\Big[\sup_{s\in[t,T]}|\psi^n(s)|^p\Big]+\sup_{t\in[0,T]}\dbE_t\Big[\sup_{s\in[t,T]}\int_s^T|g_2^n(s,r,0,0,0)|^pdr\Big]\Big\}\\
                                            \nn &\q +C_p\sup_{s\in[0,T]}\Big[\int_s^T|g_1^n(s,r,0,0,0)| dr\Big]^p\\
                                         \nn &\leq  C_p\dbE\Big\{\sup_{t\in[0,T]}\dbE_t
                                            \Big[\sup_{s\in[0,T]}|\psi^n(s)|^p\Big]+\sup_{t\in[0,T]}\dbE_t\Big[\sup_{s\in[0,T]}\int_s^T|g_2^n(s,r,0,0,0)|^p dr\Big]\Big\}\\
                                           \nn &\q +C_p\sup_{s\in[0,T]}\Big[\int_s^T|g_1^n(s,r,0,0,0)| dr\Big]^p\\
                                         \nn&\leq  C_p\dbE \Big[\sup_{t\in[0,T]}|\psi^n(t)|^p\Big]+C_p\dbE\Big[\sup_{t\in[0,T]}\int_t^T|g_2^n(t,s,0,0,0)|^p ds\Big]\\
                                             \nn &\q +C_p\sup_{t\in[0,T]}\Big[\int_t^T|g_1^n(t,s,0,0,0)| ds\Big]^p\\
                                     \nn &\leq  C_p\dbE \Big[\sup_{t\in[0,T]}|\psi(t)|^p\Big]+ C_p\dbE\Big[\sup_{t\in[0,T]}\int_t^T|g_2(t,s,0,0,0)|^p ds\Big]\\
\label{EBSVIEnt-E-estimate}                      &\q +C_p\sup_{t\in[0,T]}\Big[\int_t^T|g_1(t,s,0,0,0)| ds\Big]^p<\i.
\end{align}
Further, similar to the above \rf{EBSVIEnt-E-estimate}, we have
\begin{align}
\nn&\dbE\Big[\sup_{t\in[0,T]}|Y^m(t,t)-Y^n(t,t)|^p\Big] \\
\nn&\q\leq C_p\dbE\Big[\sup_{t\in[0,T]}\int_t^T|g_2^m(t,s,Y^m(t,s),Y^m(s,s),Z^m(t,s))-g_2^n(t,s,Y^m(t,s),Y^m(s,s),Z^m(t,s))|^p ds\Big]\\
\nn&\qq +C_p\Big[\sup_{t\in[0,T]}\int_t^T|g_1^m(t,s,Y^m(t,s),Y^m(s,s),Z^m(t,s))-g_1^n(t,s,Y^m(t,s),Y^m(s,s),Z^m(t,s))| ds\Big]^p\\
\label{EBSVIEmnt-E-estimate}&\qq +C_p\dbE\Big[\sup_{t\in[0,T]}|\psi^m(t)-\psi^n(t)|^p\Big].
\end{align}
It is worth pointing out that
\begin{align*}
&g_1^m(t,s,Y^m(t,s),Y^m(s,s),Z^m(t,s))-g_1^n(t,s,Y^m(t,s),Y^m(s,s),Z^m(t,s))\\
&\q=g_1^m(t,s,0,0,0)-g_1^n(t,s,0,0,0),\q (t,s)\in\D[0,T]
\end{align*}
is a bounded deterministic function and
\begin{align*}
&g_2^m(t,s,Y^m(t,s),Y^m(s,s),Z^m(t,s))-g_2^n(t,s,Y^m(t,s),Y^m(s,s),Z^m(t,s))\\
&\q=g_2^m(t,s,0,0,0)-g_2^n(t,s,0,0,0),\q (t,s)\in\D[0,T]
\end{align*}
is bounded.
By the definitions of $\psi^n(\cd)$, $g^n(\cd)$ and dominated convergence theorem [$|\psi^n(\cd)|\leq |\psi(\cd)|$, $|g_i^n(t,s,0,0,0)|\leq|g_i(t,s,0,0,0)|,i=1,2$], we have
\begin{align}
\nn&\dbE\Big[\sup_{t\in[0,T]}|Y^m(t,t)-Y^n(t,t)|^p\Big] \\
\nn&\q\leq C_p\dbE\Big[\sup_{t\in[0,T]}\int_t^T|g_2^m(t,s,0,0,0)-g_2^n(t,s,0,0,0)|^p ds\Big]+C_p\dbE\Big[\sup_{t\in[0,T]}|\psi^m(t)-\psi^n(t)|^p\Big]\\
\nn&\qq +C_p \Big[\sup_{t\in[0,T]}\int_t^T|g_1^m(t,s,0,0,0)-g_1^n(t,s,0,0,0)| ds\Big]^p\\
\label{EBSVIEmnt-E-con}&\q\to 0, ~\hbox{as}~ m,n\to\i,
\end{align}
which implies that $\{Y^n(t,t);0\leq t\leq T\}_{n\geq 0}$ is  Cauchy in $L^p(\Om;C([0,T];\dbR^m))$.
Further, by \rf{EBSVIE-sta-estimate} in \autoref{thm-wellposedness}  and dominated convergence theorem, it is clear to see that
\bel{EBSVIEmy-E-conv}
\lim_{n\to\i}\sup_{t\in[0,T]}\dbE\Big[|Y^n(t,t)-Y(t,t)|^p\Big]\leq\lim_{n\to\i}\sup_{t\in[0,T]}\dbE\Big[\sup_{s\in[t,T]}|Y^n(t,s)-Y(t,s)|^p\Big]=0.
\ee
Combining \rf{EBSVIEmnt-E-con} with \rf{EBSVIEmy-E-conv}, we have
\bel{EBSVIEmy-E-conv2}
\lim_{n\to\i}\dbE\Big[\sup_{t\in[0,T]}|Y^n(t,t)-Y(t,t)|^p\Big]= 0.
\ee
Combining \rf{EBSVIEmy-E-conv2} with \rf{EBSVIEnt-E-estimate}, we have the estimate  \rf{pro-EBSVIE-supestimate}.
\end{proof}
\begin{remark}\rm
\autoref{thm-y-continuity} gives the $L^p(\Om;C([0,T];\dbR^m))$-norm estimate of $Y(t,t);0\leq t\leq T$,
which plays a basic role in our subsequent analysis.

\end{remark}

\section{Regularity of the adapted solution}\label{Sec:Regularity}

In this section, we are going to discuss the regularity property of the adapted solution to EBSVIE \rf{ebsvie-I}.
To begin with, we introduce the following space:
For any $p\geq 2$ and $0\leq R< S\leq T$, let $\Psi^p[R,S]$ be the space consists of all processes $\psi(\cd)\in L^\i(R,S;L^p_{\cF_T}(\Om;\dbR^m))$ such that
$$
\|\psi(\cd)\|^p_{\Psi^p[R,S]}\deq\sup_{(t,s)\in[R,S]^2}\dbE\Big[|\psi(t)|^p+\sum_{i=1}^d|D^i_t\psi(s)|^p\Big]<\i,
$$
where $D^i_t\psi(s)$ is the  Malliavin derivative of $\psi(s)$ with respect to $W^i(\cd)$.

\ms
Now, we introduce the following assumption [comparing with \ref{ass:B1}].
\begin{taggedassumption}{(B.3)}\label{ass:B3}\rm
Let $p\geq 2$ and the generator $g:\D\times \dbR^m\times\dbR^m\times\dbR^{m\times d}\times\Om\to\dbR^m$ be $\cB(\D\times \dbR^m\times\dbR^m\times\dbR^{m\times d})\bigotimes\cF_T$-measurable,
with $s\mapsto g(t,s,y,y',z)$ being $\dbF$-progressively measurable for all $(t,y,y',z)\in [0,T)\times\dbR^m\times\dbR^m\times\dbR^{m\times d}$.
Let $(y,y',z)\mapsto g(t,s,y,y',z)$ be continuously differentiable, and $(y,y',z)\mapsto [D^i_r g](t,s,y,y',z)$ be continuous.
Moreover, there is a process $L_0(t,s):\D[0,T]\times\Om\to[0,\i)$ satisfying
$$
\sup_{t\in[0,T]}\dbE\Big(\int_t^T|L_0(t,s)|ds\Big)^p<\i,
$$
such that
$$
\sum_{i=1}^d\big|[D^i_r g](t,s,y,y',z)\big|\leq L_0(t,s),\q \forall (t,s,y,y',z)\in\D[0,T]\times\dbR^m\times\dbR^m\times\dbR^{m\times d}.
$$
\end{taggedassumption}
The first main result of this section is the following.
\begin{theorem}\label{thm-EBSVIE-malliavan}
Let {\rm \ref{ass:B1}} and {\rm \ref{ass:B3}} hold.
For any $\psi(\cd)\in\Psi^p[0,T] $, let $(Y(\cd,\cd),Z(\cd,\cd))\in\cH^p[0,T]$ be the unique adapted solution to EBSVIE \rf{ebsvie-I}.
For any $(t,s)\in\D[0,T]$, $(Y(t,s),Z(t,s))$ is Malliavan derivable,
and $\{(D_r Y(t,s),D_r Z(t,s));(t,r)\in[0,T]^2;s\in[r\vee t,T]\}$ is the unique adapted solution to the following EBSVIE:
\begin{align}
\nn D^i_r Y(t,s)&=D^i_r \psi(t)+\int_s^T \bigg\{[D^i_rg] (t,\t,Y(t,\t),Y(\t,\t),Z(t,\t))\\
\nn&\qq\qq\qq\qq+g_y (t,\t,Y(t,\t),Y(\t,\t),Z(t,\t))D^i_r Y(t,\t)\\
\nn&\qq\qq\qq\qq+g_{y'} (t,\t,Y(t,\t),Y(\t,\t),Z(t,\t))D^i_r Y(\t,\t)\\
\nn&\qq\qq\qq\qq+\sum_{j=1}^d g_{z_j} (t,\t,Y(t,\t),Y(\t,\t),Z(t,\t))D^i_r Z_j(t,\t)\bigg\}d\t\\
\label{thm-reg-ma1}&\qq\qq-\int_s^T D^i_r Z(t,\t)dW(\t),\qq t\in[0,T],~s\in[r\vee t,T],~1\leq i\leq d.
\end{align}
In addition,
\begin{align}
\nn Z_i(t,r)&=D^i_r \psi(t)+\int_r^T \bigg\{[D^i_rg] (t,\t,Y(t,\t),Y(\t,\t),Z(t,\t))\\
\nn&\qq\qq\qq\qq+g_y (t,\t,Y(t,\t),Y(\t,\t),Z(t,\t))D^i_r Y(t,\t)\\
\nn&\qq\qq\qq\qq+g_{y'} (t,\t,Y(t,\t),Y(\t,\t),Z(t,\t))D^i_r Y(\t,\t)\\
\nn&\qq\qq\qq\qq+\sum_{j=1}^d g_{z_j} (t,\t,Y(t,\t),Y(\t,\t),Z(t,\t))D^i_r Z_j(t,\t)\bigg\}d\t\\
\label{thm-reg-ma2}&\qq\qq-\int_r^T D^i_r Z(t,\t)dW(\t),\qq (t,r)\in\D[0,T],~1\leq i\leq d,
\end{align}
where $Z_i(t,r)$ denotes the $i$-th column of the matrix $Z(t,r)$.
\end{theorem}
\begin{proof}
We see from the proof of \autoref{thm-wellposedness} that when $T-S>0$ is small, the map $\Th(\cd,\cd)$ defined by \rf{Th} is a contraction on $\cH^p[S,T]$.
Therefore, a Picard iteration sequence converges to the unique solution. Namely, if we define
\bel{Picard}\left\{\begin{aligned}
   (Y^0(\cd),Z^0(\cd)) &= 0,\\
   (Y^{k+1}(\cd),Z^{k+1}(\cd)) &= \Th(Y^{k}(\cd),Z^{k}(\cd)), ~k\geq 0,
\end{aligned}\right.\ee
then
\bel{Picard-converge}
\lim_{k\to\i} \|(Y^{k}(\cd),Z^{k}(\cd))-(Y(\cd),Z(\cd))\|_{\cH^p[S,T]}=0.
\ee
Next, from
$$ Y^{k+1}(t,s)=\psi(t) + \int_s^T g(t,r,Y^{k+1}(t,r),Y^k(r,r),Z^{k+1}(t,r))dr - \int_s^T Z^{k+1}(t,r)dW(r),$$
similar to \cite[Proposition 2.2]{Pardoux--Peng 1992}, we can recursively show
$$(D^i_rY^{k}(\cd,\cd),D^i_rZ^{k}(\cd,\cd))\in\cH^p[S,T],~k\geq 0,$$
and
\begin{align*}
D^i_r Y^{k+1}(t,s)&=D^i_r \psi(t)+\int_s^T \bigg\{[D^i_rg] (t,\t,Y^{k+1}(t,\t),Y^k(\t,\t),Z^{k+1}(t,\t))\\
&\qq\qq\qq\qq+g_y (t,\t,Y^{k+1}(t,\t),Y^k(\t,\t),Z^{k+1}(t,\t))D^i_r Y^{k+1}(t,\t)\\
&\qq\qq\qq\qq+g_{y'} (t,\t,Y^{k+1}(t,\t),Y^k(\t,\t),Z^{k+1}(t,\t))D^i_r Y^k(\t,\t)\\
&\qq\qq\qq\qq+\sum_{j=1}^d g_{z_j} (t,\t,Y^{k+1}(t,\t),Y^k(\t,\t),Z^{k+1}(t,\t))D^i_r Z^{k+1}_j(t,\t)\bigg\}d\t\\
&\qq\qq-\int_s^T D^i_r Z^{k+1}(t,\t)dW(\t),\qq 1\leq i\leq d.
\end{align*}
Next, we introduce the following EBSVIE [which is a formal Malliavin differentiation of \rf{ebsvie-I} ]:
\begin{align*}
\hat Y^{r,i}(t,s)&=D^i_r \psi(t)+\int_s^T \bigg\{[D^i_rg] (t,\t,Y(t,\t),Y(\t,\t),Z(t,\t))\\
&\qq\qq\qq\qq+g_y (t,\t,Y(t,\t),Y(\t,\t),Z(t,\t)) \hat Y^{r,i}(t,\t)\\
&\qq\qq\qq\qq+g_{y'} (t,\t,Y(t,\t),Y(\t,\t),Z(t,\t))\hat Y^{r,i}(\t,\t)\\
&\qq\qq\qq\qq+\sum_{j=1}^d g_{z_j} (t,\t,Y(t,\t),Y(\t,\t),Z(t,\t))\hat Z^{r,i}_j(t,\t)\bigg\}d\t\\
&\qq-\int_s^T \hat Z^{r,i}(t,\t)dW(\t),\qq 1\leq i\leq d.
\end{align*}
Then, by  the stability estimate \rf{EBSVIE-sta-estimate} in \autoref{thm-wellposedness}, we have
\begin{align}
\nn\th_{k+1}&\deq\|(D^i_r Y^{k+1}(\cd),D^i_r Z^{k+1}(\cd))-(\hat Y^{r,i}(\cd),\hat Z^{r,i}(\cd))\|^p_{\cH^p[S,T]}\\
        &\nn\leq C_p\sup_{t\in[S,T]}\dbE\Big\{\int_t^T \big|[D^i_rg] (t,\t,Y(t,\t),Y(\t,\t),Z(t,\t))\\
        &\nn\qq\qq\qq\qq\q-[D^i_rg] (t,\t,Y^{k+1}(t,\t),Y^k(\t,\t),Z^{k+1}(t,\t))\big|\\
        &\nn\qq\qq\qq\qq+\big|g_y(t,\t,Y(t,\t),Y(\t,\t),Z(t,\t))\\
        &\nn\qq\qq\qq\qq\q-g_y(t,\t,Y^{k+1}(t,\t),Y^k(\t,\t),Z^{k+1}(t,\t))\big|\big|\hat Y^{r,i}(t,\t)\big|\\
        &\nn\qq\qq\qq\qq+\big|g_{y^\prime}(t,\t,Y(t,\t),Y(\t,\t),Z(t,\t))\\
        &\nn\qq\qq\qq\qq\q-g_{y^\prime}(t,\t,Y^{k+1}(t,\t),Y^k(\t,\t),Z^{k+1}(t,\t))\big|\big|\hat Y^{r,i}(\t,\t)\big|\\
        &\nn\qq\qq\qq\qq+\sum_{j=1}^d\big|g_{z_j}(t,\t,Y(t,\t),Y(\t,\t),Z(t,\t))\\
        &\nn\qq\qq\qq\qq\q-g_{z_j}(t,\t,Y^{k+1}(t,\t),Y^k(\t,\t),Z^{k+1}(t,\t))\big|\big|\hat Z_j^{r,i}(t,\t)\big|ds\Big\}^p\\
        &\nn\q +C_p (T-S)^p\|(D^i_r Y^k(\cd),D^i_r Z^k(\cd))-(\hat Y^{r,i}(\cd),\hat Z^{r,i}(\cd))\|^p_{\cH^p[S,T]}\\
        &\label{thetak}\deq \eta_k+\a\th_k.
\end{align}
If necessary, we shrink $T-S$ such that
\bel{alpha}
\a\deq C_p(T-S)^p<1.
\ee
By the convergence \rf{Picard-converge} and dominated convergence theorem, we see that
\bel{eta-con}
\lim_{k\to\i}\eta_k=0.
\ee
Then \rf{thetak} implies
\bel{theta-con}
\lim_{k\to\i}\th_k=0.
\ee
Since operator $D^i_r$ is closed, we have
\bel{hatY-Yk}
\hat Y^{r,i}(t,s)=D^i_rY(t,s),\q \hat Z^{r,i}(t,s)=D^i_rZ(t,s),\qq (t,s)\in\D[S,T],~a.s.
\ee
This proves \rf{thm-reg-ma1}--\rf{thm-reg-ma2} for $T-S$ small.
Similar to the proof of \autoref{thm-wellposedness}, we can prove \rf{thm-reg-ma1}--\rf{thm-reg-ma2} on $[0,T]$ by induction.
\end{proof}

Now, let us recall the notation $X^{t,x}(s)\deq X^{t,x}(s\vee t);0\leq s\leq T$ and consider the Markovian EBSVIE \rf{mebsvie}.
We first introduce the following assumption [comparing with \ref{ass:B1}--\ref{ass:B3}].
\begin{taggedassumption}{(B.4)}\label{ass:B4}
Suppose the generator $g(\cd)$ and the free term $\psi(\cd)$ satisfy:
\begin{enumerate}[(i)]
\item
Let the generator $g:\D[0,T]\times\dbR^d\times\dbR^m\times\dbR^m\times\dbR^{m\times d}\to\dbR^m$ be
continuous such that
\bel{B4-g(t,s)}
\sup_{t\in[0,T]}\int_t^T|g(t,s,0,0,0,0)|ds<\i.
\ee
Moreover, for any $(t,s)\in\D[0,T]$, let $g(t,s,\cd,\cd,\cd,\cd)\in C^3_b(\dbR^d\times\dbR^m\times\dbR^m\times\dbR^{m\times d};\dbR^{m})$
and the corresponding partial derivatives of order less than or equal to three are bounded by a constant $L>0$.
\item
Let the free term $\psi:[0,T]\times\dbR^d\to\dbR^m$ be continuous such that
\bel{B4-psi(t,x)}
\sup_{t\in[0,T]}|\psi(t,0)|<\i.
\ee
For any $t\in[0,T]$, let $\psi(t,\cd)\in C^3_b(\dbR^d;\dbR^{m})$ and the corresponding partial derivatives of order less than or equal to three are bounded by the constant $L>0$.
\end{enumerate}
\end{taggedassumption}

\begin{corollary}\label{Mallivan-MEBSVIE}
Let $p\geq 2$ and {\rm \ref{ass:F.1}}--{\rm \ref{ass:B4}} hold,
then EBSVIE \rf{mebsvie} admits a unique adapted solution $(Y^{t,x}(\cd,\cd),Z^{t,x}(\cd,\cd))\in\cH^p[0,T]$.
For any $(s,r)\in\D[0,T]$, $(Y^{t,x}(s,r),Z^{t,x}(s,r))$ is Malliavan derivable,
and $\{(D_\t Y^{t,x}(s,r),D_\t Z^{t,x}(s,r));(s,\t)\in[0,T]^2;r\in[s\vee \t,T]\}$ is the unique adapted solution to the following EBSVIE:
\begin{align}
\nn&\q D^i_\t Y^{t,x}(s,r)\\
\nn &= \psi_x(s,X^{t,x}(T))D^i_\t X^{t,x}(T)\\
\nn&\q+\int_r^T \bigg\{g_x(s,u,X^{t,x}(u),Y^{t,x}(s,u),Y^{t,x}(u,u),Z^{t,x}(s,u))D^i_\t X^{t,x}(u)\\
\nn &\q\qq+g_y(s,u,X^{t,x}(u),Y^{t,x}(s,u),Y^{t,x}(u,u),Z^{t,x}(s,u))D^i_\t Y^{t,x}(s,u)\\
\nn &\q\qq+g_{y'} (s,u,X^{t,x}(u),Y^{t,x}(s,u),Y^{t,x}(u,u),Z^{t,x}(s,u))D^i_\t Y^{t,x}(u,u)\\
\nn &\q\qq+\sum_{j=1}^d g_{z_j} (s,u,X^{t,x}(u),Y^{t,x}(s,u),Y^{t,x}(u,u),Z^{t,x}(s,u))D^i_\t Z^{t,x}_j(s,u)\bigg\}du\\
\label{Mallivan-MEBSVIE-Y}&\q-\int_r^T D^i_\t Z^{t,x}(s,u)dW(u),\qq s\in[0,T],~r\in[s\vee\t,T],~1\leq i\leq d.
\end{align}
Moreover, for any $1\leq i\leq d$, $\{D^i_r Y^{t,x}(s,r);(s,r)\in\D[0,T]\}$ is a version of $\{ Z_i^{t,x}(s,r);(s,r)\in\D[0,T]\}$,
where $Z_i^{t,x}(\cd)$ denotes the $i$-th column of the matrix  $Z^{t,x}(\cd)$.
\end{corollary}

\begin{proof}
By \autoref{lmm:well-posedness-SDE} and {\rm \ref{ass:B4}}, we have
\begin{align*}
 &\sup_{s\in[0,T]}\dbE|\psi(s,X^{t,x}(T))|^p\leq C_p(1+|x|^p)<\i, \\
  &\sup_{s\in[0,T]}\dbE\Big[\int_s^T|g(s,r,X^{t,x}(r),0,0,0)|dr\Big]^p\leq C_p(1+|x|^p)<\i.
\end{align*}
Thus, by \autoref{thm-wellposedness}, EBSVIE \rf{mebsvie} admits a unique adapted solution $(Y^{t,x}(\cd,\cd),Z^{t,x}(\cd,\cd))\in\cH^p[0,T]$.
Further, by {\rm \ref{ass:B4}}, \autoref{lmm:Malliavan-SDE} and \autoref{remark-sde}, we have
\begin{align*}
&\sup_{s\in[0,T]}\dbE|D^i_r\psi(s,X^{t,x}(T))|^p=\sup_{s\in[0,T]}\dbE|\psi_x(s,X^{t,x}(T))D^i_r X^{t,x}(T)|^p\leq C_p(1+|x|^p)<\i,\\
&\sup_{s\in[0,T]}\dbE\Big[\int_s^T |D^i_rg(s,u,X^{t,x}(u),y,y',z)|du\Big]^p\\
&\q=\sup_{s\in[0,T]}\dbE\Big[\int_s^T |g_x(s,u,X^{t,x}(u),y,y',z)D^i_r X^{t,x}(u)|du\Big]^p\\
&\q\leq \sup_{s\in[0,T]}\dbE\Big[\int_s^T |D^i_r X^{t,x}(u)|du\Big]^p
\leq C_p(1+|x|^p)<\i.
\end{align*}
Thus, by \autoref{thm-EBSVIE-malliavan}, for any $(s,r)\in\D[0,T]$, $(Y^{t,x}(s,r),Z^{t,x}(s,r))$ is Malliavan derivable,
and $\{(D_\t Y^{t,x}(s,r),D_\t Z^{t,x}(s,r));(s,\t)\in[0,T]^2;r\in[s\vee \t,T]\}$ is the unique adapted solution to EBSVIE \rf{Mallivan-MEBSVIE-Y}.
\end{proof}

Let $\{\nabla Y^{t,x}(s,r),\nabla Z^{t,x}(s,r);(s,r)\in\D[0,T]\}\in\cH^p[0,T]$ be the unique adapted solution to the following EBSVIE:
\begin{align}
\nn&\q \nabla Y^{t,x}(s,r)\\
\nn &= \psi_x(s,X^{t,x}(T))\nabla X^{t,x}(T)\\
\nn&\q+\int_r^T \bigg\{g_x(s,u,X^{t,x}(u),Y^{t,x}(s,u),Y^{t,x}(u,u),Z^{t,x}(s,u))\nabla X^{t,x}(u)\\
\nn &\q\qq+g_y(s,u,X^{t,x}(u),Y^{t,x}(s,u),Y^{t,x}(u,u),Z^{t,x}(s,u))\nabla Y^{t,x}(s,u)\\
\nn &\q\qq+g_{y'}(s,u,X^{t,x}(u),Y^{t,x}(s,u),Y^{t,x}(u,u),Z^{t,x}(s,u))\nabla Y^{t,x}(u,u)\\
\nn &\q\qq+\sum_{j=1}^d g_{z_j}(s,u,X^{t,x}(u),Y^{t,x}(s,u),Y^{t,x}(u,u),Z^{t,x}(s,u))\nabla Z^{t,x}_j(s,u)\bigg\}du\\
\label{diff-MEBSVIE-Y}&\q-\int_r^T \nabla Z^{t,x}(s,u)dW(u),\qq (s,r)\in\D[0,T].
\end{align}
Similar to \autoref{remark-sde}, we have the following proposition.

\begin{proposition}\label{pro-Mallivan-diff-MEBSVIE}
For any $(t,x)\in[0,T)\times\dbR^d$, $s\in[0,T]$, $r\in[t\vee s,T]$, $\t\in[t,r]$,

\bel{pro-Mallivan-diff-MEBSVIE-main}
D_\t Y^{t,x}(s,r)=\nabla Y^{t,x}(s,r)(\nabla X^{t,x}(\t))^{-1}\si(\t,X^{t,x}(\t)),
\ee
and the process $\{ D_r Y^{t,x}(s,r);r\in[t\vee s, T]\}$ is a.s. continuous.
\end{proposition}

\begin{proof}
By \rf{diff-malli-sde} in \autoref{remark-sde} and the uniqueness of the adapted solution to \rf{diff-MEBSVIE-Y}, we have \rf{pro-Mallivan-diff-MEBSVIE-main},
and the continuity of $\{ D_r Y^{t,x}(s,r);r\in[t\vee s, T]\}$ comes from that of $\nabla Y^{t,x}(s,r)$, $ (\nabla X^{t,x}(r))^{-1}$, $\si(r, X^{t,x}(r));r\in[t\vee s, T]$.
\end{proof}

For any $(t,s,x)\in[0,T)^2\times\dbR^d$, by \autoref{Mallivan-MEBSVIE} and \autoref{pro-Mallivan-diff-MEBSVIE}, we deduce that $\{Z^{t,x}(s,r);r\in[t\vee s, T]\}$ has an a.s. continuous version,
and we shall identify  $Z^{t,x}(s,\cd)$ with its continuous version from now on.
An immediate consequence of \autoref{Mallivan-MEBSVIE} and \autoref{pro-Mallivan-diff-MEBSVIE} is now:

\begin{lemma}\label{Z(t,s)-repre}
For any $(t,x)\in[0,T)\times\dbR^d$, $s\in[0,T]$, $r\in[t\vee s,T]$, we have
$$
Z^{t,x}(s,r)=\nabla Y^{t,x}(s,r)(\nabla X^{t,x}(r))^{-1}\si(r,X^{t,x}(r)).
$$
\end{lemma}
For any $p\geq 2$, similar to \autoref{thm-wellposedness}, we can establish the $L^p(\Om)$ estimate for $\sup_{r\in[s,T]}|\nabla Y^{t,x}(s,r)|$.
Thus, we deduce from the above lemma:
\begin{lemma}\label{z-estimate}
For any $(t,x)\in[0,T)\times\dbR^d$ and $p\geq 2$,  we have
$$
\sup_{s\in[t,T]}\dbE\[\sup_{r\in[s,T]}|Z^{t,x}(s,r)|^p\]<\i.
$$
\end{lemma}
Further, we have the following result.

\begin{proposition} \label{Y-adapted}
Let {\rm \ref{ass:F.1}}--{\rm \ref{ass:B4}} hold and $(Y^{t,x}(\cd,\cd),Z^{t,x}(\cd,\cd))\in\cH^p[0,T]$ be the unique adapted solution to EBSVIE \rf{mebsvie},
then for any $s\in[0,T)$, $\{(Y^{t,x}(s,r)$, $Z^{t,x}(s,r));s\leq r\leq T\}$ is $\dbF^{t}$-adapted.
\end{proposition}

\begin{proof}
Note that $X^{t,x}(\cd)$ is  $\dbF^{t}$-adapted.
Define
\bel{bm}
W^t(s)\deq W(s)-W(t),\qq t\leq s\leq T,
\ee
which is an $\dbF^{t}$-adapted Brownian motion.
Consider the following EBSVIE:
\begin{align}
\nn \tilde Y^{t,x}(s,r)&=\psi(s,X^{t,x}(T)) + \int_r^T g(s,\t,X^{t,x}(\t),\tilde Y^{t,x}(s,\t),\tilde Y^{t,x}(\t,\t),\tilde Z^{t,x}(s,\t))d\t \\
\label{adpted-ebsvie} &\qq\qq\qq-\int_r^T \tilde Z^{t,x}(s,\t)dW^t(\t),\q s\in[0,T],~r\in[s\vee t,T].
\end{align}
By \autoref{thm-wellposedness} and note that $X^{t,x}(\cd)$ is  $\dbF^{t}$-adapted, the above EBSVIE admits a unique solution $(\tilde Y^{t,x}(s,r),\tilde Z^{t,x}(s,r));~s\in[0,T],~r\in[s\vee t,T]$.
And for any $s\in[0,T]$, $(\tilde Y^{t,x}(s,\cd),\tilde Z^{t,x}(s,\cd))$ is $\dbF^{t}$-adapted.
By \rf{bm}--\rf{adpted-ebsvie}, $(\tilde Y^{t,x}(\cd,\cd),\tilde Z^{t,x}(\cd,\cd))$ also satisfies the EBSVIE \rf{mebsvie} and is also $\dbF$-adapted.
By the uniqueness of the adapted solutions to EBSVIE \rf{mebsvie},  we have
\bel{y-tily-adpted-ebsvie}
Y^{t,x}(s,r)=\tilde Y^{t,x}(s,r),\q Z^{t,x}(s,r)=\tilde Z^{t,x}(s,r),~s\in[0,T],~r\in[s\vee t,T],
\ee
which means that for any $s\in[0,T]$, $(Y^{t,x}(s,r),Z^{t,x}(s,r));r\in[s\vee t,T]$ is $\dbF^{t}$-adapted.
Further, when $s<t$, $(Y^{t,x}(s,r),Z^{t,x}(s,r));s\leq r\leq t$ is also the unique adapted solution to the following EBSVIE:
\begin{align}
\nn  Y^{t,x}(s,r)&=Y^{t,x}(s,t) + \int_r^t g(s,\t,X^{t,x}(\t), Y^{t,x}(s,\t), Y^{t,x}(\t,\t), Z^{t,x}(s,\t))d\t \\
\label{adpted-ebsvie1} &\qq\qq\qq-\int_r^t  Z^{t,x}(s,\t)dW(\t),\q (s,r)\in\D[0,t].
\end{align}
Note that $Y^{t,x}(s,t)\in\cF^t_t=\cF_0$ and  $X^{t,x}(\t)\equiv x;0\leq \t\leq t$, thus EBSVIE \rf{adpted-ebsvie1} is a deterministic integral equation,
which implies that $ Z^{t,x}(s,r)\equiv 0\in \cF_0=\cF^t_r;(s,r)\in\D[0,t]$ and $Y^{t,x}(s,r); (s,r)\in\D[0,t]$ is a deterministic function.
Combining this with \rf{y-tily-adpted-ebsvie}, we have that for any $s\in[0,T)$, $(Y^{t,x}(s,r)$, $Z^{t,x}(s,r));s\leq r\leq T$ is $\dbF^{t}$-adapted.
\end{proof}


Now, we consider the regularity of $Y^{t,x}(s,r)$.
\begin{taggedassumption}{(B.5)}\label{ass:B5}
Let the generator $g:[0,T]^2\times\dbR^d\times\dbR^m\times\dbR^m\times\dbR^{m\times d}\to\dbR^m$ and the free term $\psi:[0,T]\times\dbR^d\to\dbR^m$ satisfy {\rm \ref{ass:B4}}.
There exists a modulus of continuity $\rho:[0,\i)\to[0,\i)$ (a continuous and monotone increasing function with $\rho(0)=0$) such that
\begin{align*}
&|g(t_1,s,x,y,y',z)-g(t_2,s,x,y,,y',z)|\les\rho(|t_1-t_2|),\\
&\qq\qq\qq \forall~t_1,t_2,s\in[0,T],~(x,y,y',z)\in\dbR^d\times\dbR^m\times\dbR^m\times\dbR^{m\times d},\\
&|\psi(t_1,x)-\psi(t_2,x)|\les\rho(|t_1-t_2|),\q\forall~ t_1,t_2\in[0,T], ~x\in\dbR^d.
\end{align*}
\end{taggedassumption}
\begin{theorem}\label{theorem-Y-regular}
Let {\rm \ref{ass:F.1}} and {\rm \ref{ass:B5}} hold, then
$\{Y^{t,x}(s,r);t\in[0,T],x\in\dbR^d,(s,r)\in\D[0,T]\}$ has a version whose trajectories belong to $C^{0,0,0,2}([0,T]\times\D[0,T]\times\dbR^d;\dbR^m)$.
\end{theorem}
\begin{proof}
For any $(t,x)\in[0,T]\times\dbR^d$, let $X^{t,x}(s)=X^{t,x}(s\vee t);0\leq s\leq T$,
and $(Y^{t,x}(\cd),Z^{t,x}(\cd))\in\cH^p[0,T]$ be the unique adapted solution to EBSVIE \rf{mebsvie}  on [0,T].
By \autoref{lmm:well-posedness-SDE} and \rf {EBSVIE-estimate} in \autoref{thm-wellposedness}, we have
\begin{align}
\nn& \sup_{s\in[0,T]}\dbE\Big[\sup_{r\in[s,T]}|Y^{t,x}(s,r)|^p\Big]+\sup_{s\in[0,T]}\dbE\Big[\int_s^T|Z^{t,x}(s,r)|^2 dr\Big]^{p\over 2}\\
\nn&\q \leq  C_p\sup_{s\in[0,T]}\dbE |\psi(s,X^{t,x}(T))|^p+C_p\sup_{s\in[0,T]}\dbE\Big[\int_s^T|g(s,r,X^{t,x}(r),0,0,0)|dr\Big]^p\\
\nn&\q \leq  C_p\dbE\Big[1+\sup_{s\in[0,T]}|X^{t,x}(s)|^p\Big] \\
\label{thm-regular-estimate-y-z}&\q \leq C_p(1+|x|^p).
\end{align}
Further, let
\begin{align}
g_1(s,r,y,y',z)&=g(s,r,0,y,y',z),\\
g_2(s,r,y,y',z)&=g(s,r,X^{t,x}(r),y,y',z)-g(s,r,0,y,y',z).
\end{align}
Note that $g(\cd)=g_1(\cd)+g_2(\cd)$ and
\begin{align}
\nn&\dbE\Big [\sup_{s\in[0,T]}|\psi(s,X^{t,x}(T))|^p\Big]+\sup_{s\in[0,T]}\int_s^T|g_1(s,r,0,0,0)|dr
+\dbE\Big[\sup_{s\in[0,T]}\int_s^T|g_2(s,r,0,0,0)|^pdr\Big]\\
\nn&\q=\dbE\Big [\sup_{s\in[0,T]}|\psi(s,X^{t,x}(T))|^p\Big]+\Big[\sup_{s\in[0,T]}\int_s^T|g(s,r,0,0,0,0)|dr\Big]^p\\
\nn& \qq+\dbE\Big[\sup_{s\in[0,T]}\int_s^T|g(s,r,X^{t,x}(r),0,0,0)-g(s,r,0,0,0,0)|^pdr\Big]\\
\label{thm-regular-estimate-y-z2}&\q \leq C_p\dbE\Big [1+\sup_{s\in[0,T]}|X^{t,x}(s)|^p\Big]
\leq C_p(1+|x|^p).
\end{align}
Combining this with \rf{pro-EBSVIE-supestimate} in \autoref{thm-y-continuity}, we have
\bel{thm-regular-estimate-y-z3}
\dbE\Big[\sup_{s\in[0,T]}|Y^{t,x}(s,s)|^p\Big]\leq C_p(1+|x|^p).
\ee
For any $(t',x')\in[0,T]\times\dbR^d$, define $X^{t',x'}(\cd),Y^{t',x'}(\cd),Z^{t',x'}(\cd)$ as before.
By \autoref{lmm:well-posedness-SDE}, \autoref{thm-wellposedness} and \autoref{thm-y-continuity}, similar to \rf{thm-regular-estimate-y-z} and \rf{thm-regular-estimate-y-z3}, we have
\begin{align}
\nn& \sup_{s\in[0,T]}\dbE\Big[\sup_{r\in[s,T]}|Y^{t,x}(s,r)-Y^{t',x'}(s,r)|^p\Big]+\sup_{s\in[0,T]}\dbE\Big[\int_s^T|Z^{t,x}(s,r)-Z^{t',x'}(s,r)|^2 dr\Big]^{p\over 2}\\
\nn& \q\leq  C_p\bigg\{\sup_{s\in[0,T]}\dbE\Big[\int_s^T|g(s,r,X^{t,x}(r),Y^{t,x}(s,r),Y^{t,x}(r,r),Z^{t,x}(s,r))\\
\nn&\qq\qq\qq\qq-g(s,r,X^{t',x'}(r),Y^{t,x}(s,r),Y^{t,x}(r,r),Z^{t,x}(s,r))|dr\Big]^p \\
\nn&\qq\qq+\sup_{s\in[0,T]}\dbE |\psi(s,X^{t,x}(T))-\psi(s,X^{t',x'}(T))|^p\bigg\}\\
\label{thm-regular-sta-estimate-y-z}&\q\leq C_p\dbE\Big[\sup_{0\les s\les T}|X^{t,x}(s)-X^{t^\prime,x^\prime}(s)|^p\Big]\leq C_p(1+|x|^p)(|x-x^\prime|^p+|t-t^\prime|^{{p\over 2}}),
\end{align}
and
\begin{align}
\nn& \dbE\Big[\sup_{s\in[0,T]}|Y^{t,x}(s,s)-Y^{t',x'}(s,s)|^p\Big]\\
\nn& \q\leq  C_p\bigg\{\dbE\Big[\sup_{s\in[0,T]}\int_s^T|g(s,r,X^{t,x}(r),Y^{t,x}(s,r),Y^{t,x}(r,r),Z^{t,x}(s,r))\\
\nn&\qq\qq\qq\qq-g(s,r,X^{t',x'}(r),Y^{t,x}(s,r),Y^{t,x}(r,r),Z^{t,x}(s,r))|^p dr\Big] \\
\nn&\qq\qq+\dbE\Big[\sup_{s\in[0,T]}|\psi(s,X^{t,x}(T))-\psi(s,X^{t',x'}(T))|^p\Big]\bigg\}\\
\label{thm-regular-sta-supestimate-y1}&\q\leq C_p\dbE\Big[\sup_{0\les s\les T}|X^{t,x}(s)-X^{t^\prime,x^\prime}(s)|^p\Big]\leq C_p(1+|x|^p)(|x-x^\prime|^p+|t-t^\prime|^{{p\over 2}}).
\end{align}
We note that \rf{thm-regular-sta-estimate-y-z} implies that for any
fixed $(s,x)\in[0,T)\times\dbR^d$, $ Y^{t,x}(s,t);s\leq t\leq T$ is continuous
and \rf{thm-regular-sta-supestimate-y1} implies that for any
fixed $x\in\dbR^d$, $s \mapsto Y^{s,x}(s,s)$ is continuous.
Next, for any $h\neq 0$, we consider
\begin{align}
\nn &\D_h^i Y^{t,x}(s,r)\deq h^{-1} \big[Y^{t,x+he_i}(s,r)-Y^{t,x}(s,r)\big]\\
\nn &\q =h^{-1}[\psi(s,X^{t,x+he_i}(T))-\psi(s,X^{t,x}(T))]\\
\nn &\qq+\int_r^T h^{-1}\big[g(s,\t,X^{t,x+he_i}(\t),Y^{t,x+he_i}(s,\t),Y^{t,x+he_i}(\t,\t),Z^{t,x+he_i}(s,\t))\\
\nn &\qq\qq\qq-g(s,\t,X^{t,x}(\t),Y^{t,x}(s,\t),Y^{t,x}(\t,\t),Z^{t,x}(s,\t))\big]d\t \\
\nn &\qq-\int_r^T h^{-1} \big[Z^{t,x+he_i}(s,\t)- Z^{t,x}(s,\t)\big]dW(\t)\\
\nn &\q=\int_0^1\psi_x(s,X^{t,x}(T)+\l h\D_h^iX^{t,x}(T))\D_h^iX^{t,x}(T)d\l\\
\nn &\qq+\int_r^T\int_0^1 \Big[g_x(\Xi^{t,x,h}_\l(s,\t))\D_h^iX^{t,x}(\t)+g_y(\Xi^{t,x,h}_\l(s,\t))\D_h^iY^{t,x}(s,\t)\\
\nn &\qq\qq\q +g_{y^\prime}(\Xi^{t,x,h}_\l(s,\t))\D_h^iY^{t,x}(\t,\t)+g_z(\Xi^{t,x,h}_\l(s,\t))\D_h^i Z^{t,x}(s,\t)\Big]d\l d\t \\
\label{thm-regular-sta-supestimate-Dy} &\qq -\int_r^T \D_h^iZ^{t,x}(s,\t)dW(\t),
\end{align}
where
\begin{align*}
\Xi^{t,x,h}_\l(s,\t)&=\Big(s,\t,X^{t,x}(\t)+\l h\D_h^iX^{t,x}(\t),Y^{t,x}(s,\t)+\l h\D_h^iY^{t,x}(s,\t),\\
&\qq\q Y^{t,x}(\t,\t)+\l h\D_h^iY^{t,x}(\t,\t),Z^{t,x}(s,\t)+\l h\D_h^i Z^{t,x}(s,\t)\Big),~(s,\t)\in\D[0,T].
\end{align*}
By \rf{lemma2.1.1} in \autoref{lmm:well-posedness-SDE}, \autoref{thm-wellposedness} and \autoref{thm-y-continuity},  we have
\bel{thm-regular-DY-DZ}
\sup_{s\in[0,T]}\dbE\Big[\sup_{r\in[s,T]}|\D_h^i Y^{t,x}(s,r)|^p\Big]+\sup_{s\in[0,T]}\dbE\Big[\int_s^T|\D_h^i Z^{t,x}(s,r)|^2 dr\Big]^{p\over 2}\leq C_p,
\ee
and
\bel{thm-regular-DY-DZ-sup}
\dbE\Big[\sup_{s\in[0,T]}|\D_h^i Y^{t,x}(s,s)|^p\Big]\leq C_p.
\ee
This means for any $(t,s,r)\in[0,T]\times\D[0,T]$, $Y^{t,x}(s,r)$ is differentiable in $x$.
Finally, we consider
\begin{align}
\nn &\D_h^i Y^{t,x}(s,r)-\D_{h'}^i Y^{t',x'}(s,r)\\
\nn &\q=\int_0^1\psi_x(s,X^{t,x}(T)+\l h\D_h^iX^{t,x}(T))\D_h^iX^{t,x}(T)d\l\\
\nn &\qq-\int_0^1\psi_x(s,X^{t',x'}(T)+\l h'\D_{h'}^iX^{t',x'}(T))\D_{h'}^iX^{t',x'}(T)d\l\\
\nn &\qq+\int_r^T\int_0^1 \Big[g_x(\Xi^{t,x,h}_\l(s,\t))\D_h^iX^{t,x}(\t)-g_x(\Xi^{t',x',h'}_\l(s,\t))\D_{h'}^iX^{t',x'}(\t)\Big]d\l d\t\\
\nn &\qq+\int_r^T\int_0^1 \Big[g_y(\Xi^{t,x,h}_\l(s,\t))\D_h^iY^{t,x}(s,\t)-g_y(\Xi^{t',x',h'}_\l(s,\t))\D_{h'}^iY^{t',x'}(s,\t)\Big]d\l d\t\\
\nn &\qq+\int_r^T\int_0^1 \Big[g_{y'}(\Xi^{t,x,h}_\l(s,\t))\D_h^iY^{t,x}(\t,\t)-g_{y'}(\Xi^{t',x',h'}_\l(s,\t))\D_{h'}^iY^{t',x'}(\t,\t)\Big]d\l d\t\\
\nn &\qq+\int_r^T\int_0^1 \Big[g_z(\Xi^{t,x,h}_\l(s,\t))\D_h^iZ^{t,x}(s,\t)-g_z(\Xi^{t',x',h'}_\l(s,\t))\D_{h'}^i Z^{t',x'}(s,\t)\Big]d\l d\t\\
\label{thm-regular-sta-supestimate-DDy} &\qq -\int_r^T \big[\D_h^i Z^{t,x}(s,\t)-\D_{h'}^i Z^{t',x'}(s,\t)\big]dW(\t).
\end{align}
Similar to \rf{thm-regular-sta-estimate-y-z} and \rf{thm-regular-sta-supestimate-y1}, by \autoref{lmm:well-posedness-SDE}, \autoref{thm-wellposedness} and \autoref{thm-y-continuity},
we have
\begin{align}
\nn &\sup_{s\in[0,T]}\dbE\Big[\sup_{r\in[s,T]}|\D_h^i Y^{t,x}(s,r)-\D_{h'}^i Y^{t',x'}(s,r)|^p\Big]+\dbE\Big[\sup_{s\in[0,T]}|\D_h^i Y^{t,x}(s,s)-\D_{h'}^i Y^{t',x'}(s,s)|^p\Big]\\
\label{thm-regular-DDY}&\q\leq C_p(1+|x|^p+|h|^p+|x'|^p+|h'|^p)(|x-x'|^p+|h-h'|^p+|t-t'|^{p\over 2}),
\end{align}
and
\begin{align}
\nn&\sup_{s\in[0,T]}\dbE\Big[\int_s^T|\D_h^i Z^{t,x}(s,r)-\D_{h'}^i Z^{t',x'}(s,r)|^2 dr\Big]^{p\over 2}\\
\label{thm-regular-DDZ}&\q \leq C_p(1+|x|^p+|h|^p+|x'|^p+|h'|^p)(|x-x'|^p+|h-h'|^p+|t-t'|^{p\over 2}).
\end{align}
Similar to \rf{thm-regular-DY-DZ} and \rf{thm-regular-DY-DZ-sup}, we have
\begin{align}
\nn &\sup_{s\in[0,T]}\dbE\Big[\sup_{r\in[s,T]}|\D_h^i Y^{t,x}(s,r)-\D_{h'}^i Y^{t,x'}(s,r)|^p\Big]+\dbE\Big[\sup_{s\in[0,T]}|\D_h^i Y^{t,x}(s,s)-\D_{h'}^i Y^{t,x'}(s,s)|^p\Big]\\
\label{thm-regular-DDY1}&\q\leq C_p(|x-x'|^p+|h-h'|^p),
\end{align}
and
\bel{thm-regular-DDZ1}
\sup_{s\in[0,T]}\dbE\Big[\int_s^T|\D_h^i Z^{t,x}(s,r)-\D_{h'}^i Z^{t,x}(s,r)|^2 dr\Big]^{p\over 2}
 \leq C_p(|x-x'|^p+|h-h'|^p).
\ee
This means for any $(t,s,r)\in[0,T]\times\D[0,T]$, $Y^{t,x}(s,r)$ is twice differentiable in $x$.
\end{proof}

\begin{corollary}\label{corollory-yz}
Let $\{( \nabla Y^{t,x}(s,r), \nabla Z^{t,x}(s,r));0\leq s\leq r\leq T\}$ be the unique adapted solution to EBSVIE \rf{diff-MEBSVIE-Y},
then  $\{( \nabla Y^{t,x}(s,r),  \nabla Z^{t,x}(s,r));0\leq s\leq r\leq T\}$ is the gradient of  $\{( Y^{t,x}(s,r),Z^{t,x}(s,r));0\leq s\leq r\leq T\}$ respect to $x$.
\end{corollary}
\section{EBSVIEs and Parabolic PDEs}\label{Sec:EBSVIE-PDE}
Now, we are ready to relate the EBSVIE \rf{mebsvie} to the the following systems of parabolic
partial differential equations:
\bel{FBTPDE1}\left\{\begin{aligned}
&\Th_s(t,s,x)+{1\over 2}\si(s,x)^\prime\Th_{xx}(t,s,x)\si(s,x)+\Th_x(t,s,x)b(s,x)\\
&\qq+g(t,s,x,\Th(t,s,x),\Th(s,s,x),\Th_x(t,s,x)\si(s,x))=0,\q (t,s,x)\in\D[0,T]\times\dbR^d,\\
&\Th(t,T,x)= \psi(t,x),\qq (t,x)\in[0,T]\times\dbR^d.
\end{aligned}\right.\ee
We first give a result which is similar to \cite[Theorem 3.1]{Wang--Yong 2018}.

\begin{theorem}\label{PDE-REP-BSVIE}
If $\Th(\cd,\cd,\cd)\in C^{0,1,2}(\D[0,T]\times\dbR^d;\dbR^m)$ is a classical solution of the PDEs \rf{FBTPDE1}, then
\bel{PDE-REP-BSVIE-re}
\big(Y^{t,x}(s,r),Z^{t,x}(s,r)\big)\deq\big(\Th(s,r,X^{t,x}(r)),\Th_x(s,r,X^{t,x}(r))\si(r,X^{t,x}(r))\big);~ (s,r)\in\D[ t,T]
\ee
is an adapted solution to EBSVIE \rf{mebsvie} on $[t,T]$.
\end{theorem}

\begin{proof}
For any fixed $s\in[t,T)$, using It\^{o}'s formula to $r\mapsto \Th(s,r,X^{t,x}(r))$ on $[s,T]$, we have
\begin{align*}
d\Th(s,r,X^{t,x}(r))&= \big[ \Th_r(s,r,X^{t,x}(r))+\Th_x(s,r,X^{t,x}(r))b(r,X^{t,x}(r))\\
&\q +{1\over 2}\si(r,X^{t,x}(r))^\prime\Th_{xx}(s,r,X^{t,x}(r))\si(r,X^{t,x}(r))\big]dr\\
&\q +\Th_x(s,r,X^{t,x}(r))\si(r,X^{t,x}(r))dW(r).
\end{align*}
Since $\Th$ satisfies PDE \rf{FBTPDE1}, one has
\begin{align*}
d\Th(s,r,X^{t,x}(r))&= -g(s,r,X^{t,x}(r),\Th(s,r,X^{t,x}(r)),\Th(r,r,X^{t,x}(r)),\Th_x(s,r,X^{t,x}(r))\si(r,X^{t,x}(r)))dr\\
&\q +\Th_x(s,r,X^{t,x}(r))\si(r,X^{t,x}(r))dW(r),
\end{align*}
and
\begin{equation*}
\Th(s,T,X^{t,x}(T))=\psi(s,X^{t,x}(T)).
\end{equation*}
Now, we define
$$
Y^{t,x}(s,r)=\Th(s,r,X^{t,x}(r)),\q Z^{t,x}(s,r)=\Th_x(s,r,X^{t,x}(r))\si(r,X^{t,x}(r)), \q (s,r)\in\D[t ,T].
$$
Then
\begin{align*}
Y^{t,x}(s,r)&=\psi(s,X^{t,x}(T)) + \int_r^T g(s,\t,X^{t,x}(\t),Y^{t,x}(s,\t),Y^{t,x}(\t,\t),Z^{t,x}(s,\t))d\t \\
            &\qq\qq\qq-\int_r^T Z^{t,x}(s,\t)dW(\t),
\end{align*}
which means that $(Y^{t,x}(s,r),Z^{t,x}(s,r));t\leq s\leq r\leq T$ satisfies BSVIE \rf{mebsvie}
and the desired representation \rf{PDE-REP-BSVIE-re} is obtained.
\end{proof}

We define
\bel{hatTh}
\hat\Th(t,s,x)\deq Y^{s,x}(t,s),\q (t,s,x)\in\D[0,T]\times\dbR^d.
\ee
By \autoref{Y-adapted}, $\hat\Th(\cd)$ defined by \rf{hatTh} is a deterministic function.
Now, we give the main result of this paper, which gives the converse of \autoref{PDE-REP-BSVIE}.

\begin{theorem}\label{BSVIE-REP-PDE}
Let  {\rm \ref{ass:F.1}}--{\rm \ref{ass:B4}} hold, then $\hat\Th(\cd)$ defined by \rf{hatTh}
is the unique classical solution to the system of parabolic partial differential equations \rf{FBTPDE1}.
\end{theorem}
\begin{proof}
By \autoref{theorem-Y-regular}, $\{Y^{s,x}(t,s);(t,s)\in\D[0,T],x\in\dbR^d\}\in C^{0,0,2}(\D[0,T]\times\dbR^d;\dbR^d)$.
By \rf{thm-regular-estimate-y-z}--\rf{thm-regular-estimate-y-z3} and \rf{hatTh}, we have
\bel{th-th}
|\hat\Th(t,s,x)|,~|\hat\Th(s,s,x)|\leq C_p(1+|x|).
\ee
By \rf{thm-regular-DY-DZ}--\rf{thm-regular-DY-DZ-sup}--\rf{thm-regular-DDY1} and \rf{hatTh}, we have
\bel{thx-thxx}
|\hat\Th_x(t,s,x)|,~|\hat\Th_{xx}(t,s,x)|\leq C_p.
\ee
For any $(t,s)\in\D[0,T]$, $x\in\dbR^d$, let $h>0$ be such that $s+h\leq T$.
Clearly, $Y^{s,x}(t,s+h)=Y^{s+h,X^{s,x}(s+h)}(t,s+h)$. Hence, we have
\begin{align*}
&\hat\Th(t,s+h,x)-\hat\Th(t,s,x)\\
&\q =\hat\Th(t,s+h,x)-\hat\Th(t,s+h,X^{s,x}(s+h))+\hat\Th(t,s+h,X^{s,x}(s+h))-\hat\Th(t,s,x)\\
&\q =\hat\Th(t,s+h,x)-\hat\Th(t,s+h,X^{s,x}(s+h))+Y^{s,x}(t,s+h)-Y^{s,x}(t,s)\\
&\q =-\int_s^{s+h}\big[\hat\Th_x(t,s+h,X^{s,x}(r))b(r,X^{s,x}(r))+{1\over 2}\si(r,X^{s,x}(r))^\prime\hat\Th_{xx}(t,s+h,X^{s,x}(r))\si(r,X^{s,x}(r))\big]dr\\
&\q {\hp = }-\int_s^{s+h}\hat\Th_x(t,s+h,X^{s,x}(r))\si(r,X^{s,x}(r))dW(r)\\
&\q {\hp = }- \int_s^{s+h} g(t,r,X^{s,x}(r),Y^{s,x}(t,r),Y^{s,x}(r,r),Z^{s,x}(t,r))dr +\int_s^{s+h} Z^{s,x}(t,r)dW(r).
\end{align*}
Let $s=s_0<s_1<...<s_n=T$, we have
\begin{align}
\nn&\psi(t,x)-\hat\Th(t,s,x)\\
\nn&\q =-\sum_{i=0}^{n-1}\int_{s_i}^{s_{i+1}}\big[\hat\Th_x(t,s_{i+1},X^{s_i,x}(r))b(r,X^{s_i,x}(r))\\
\nn&\qq\qq\qq+{1\over 2}\si(r,X^{s_i,x}(r))^\prime\hat\Th_{xx}(t,s_{i+1},X^{s_i,x}(r))\si(r,X^{s_i,x}(r))\\
\nn&\qq\qq\qq +g(t,r,X^{s_i,x}(r),Y^{s_i,x}(t,r),Y^{s_i,x}(r,r),Z^{s_i,x}(t,r))\big]dr\\
\label{psi-th}&\q {\hp = }-\sum_{i=0}^{n-1}\int_{s_i}^{s_{i+1}}\big[\hat\Th_x(t,s_{i+1},X^{s_i,x}(r))\si(r,X^{s_i,x}(r))-Z^{s_i,x}(t,r)\big]dW(r).
\end{align}
For any $p>2$, $0<\e<{1\over 2}-{1\over p}$, and the fixed $(t,x)\in[0,T]\times\dbR^d$, by \autoref{lmm:well-posedness-SDE}, \rf{thm-regular-sta-estimate-y-z}--\rf{thm-regular-sta-supestimate-y1} and Kolmogorov continuity theorem \cite[Theorem 3.1]{Friz 2014},
there is a  random variable $K(\om)\in L^p(\Om;\dbR)$ such that
\begin{align}
 \nn\sup_{r\in[0,T]}|X^{s,x}(r)-X^{s',x}(r)|\leq K(\om)|s-s'|^{{1\over2}- {1\over p}-\e},\\
  \label{x-y-y'}\sup_{r\in[0,T]}|Y^{s,x}(t,r)-Y^{s',x}(t,r)|\leq K(\om)|s-s'|^{{1\over2}- {1\over p}-\e},\\
\nn\sup_{r\in[0,T]}|Y^{s,x}(r,r)-Y^{s',x}(r,r)|\leq K(\om)|s-s'|^{{1\over2}- {1\over p}-\e}.
\end{align}
Thus,
\begin{align}
\nn &|X^{s,x}(r)-x|=|X^{s,x}(r)-X^{r,x}(r)| \leq \sup_{\t\in[0,T]}|X^{s,x}(\t)-X^{r,x}(\t)|
\leq K(\om)|s-r|^{{1\over2}- {1\over p}-\e},\\
\nn &|Y^{s,x}(t,r)-\hat\Th(t,r,x)|=|Y^{s,x}(t,r)-Y^{r,x}(t,r)| \\
\label{x-x'}&\q\leq \sup_{\t\in[0,T]}|Y^{s,x}(t,\t)-Y^{r,x}(t,\t)|\leq K(\om)|s-r|^{{1\over2}- {1\over p}-\e},\\
\nn &|Y^{s,x}(r,r)-\hat\Th(r,r,x)|=|Y^{s,x}(r,r)-Y^{r,x}(r,r)| \\
\nn&\q\leq \sup_{\t\in[0,T]}|Y^{s,x}(\t,\t)-Y^{r,x}(\t,\t)|\leq K(\om)|s-r|^{{1\over2}- {1\over p}-\e}.
\end{align}
By \autoref{Z(t,s)-repre}, \rf{thm-regular-DDY},  \autoref{lmm:well-posedness-SDE}, and Kolmogorov continuity theorem, we have
\begin{align}
 \nn&\sup_{r\in[0,T]}|Z^{s,x}(t,r)-Z^{s',x}(t,r)|\\
 \nn& \q =\sup_{r\in[0,T]}|\nabla Y^{s,x}(t,r)(\nabla X^{s,x}(r))^{-1}\si(r,X^{s,x}(r))-\nabla Y^{s',x}(t,r)(\nabla X^{s',x}(r))^{-1}\si(r,X^{s',x}(r))|\\
\label{z-z'} &\q\leq K(\om)|s-s'|^{{1\over2}- {1\over p}-\e}.
\end{align}
Similar to \rf{x-x'}, we have
\bel{z-z'1}
 |Z^{s,x}(t,r)-\hat\Th_x(t,r,x)\si(r,x)|\leq K(\om)|s-r|^{{1\over2}- {1\over p}-\e}.
\ee
Combining \rf{th-th}--\rf{thx-thxx}, \rf{x-x'} with \rf{z-z'1}, by dominated  convergence theorem,
let
$$\lim_{n\to\i}\sup_{0\leq i\leq n-1}|s_{i+1}-s_i|=0,$$
we obtain in the limit:
\begin{align*}
&\hat\Th(t,s,x) =\psi(t,x)+\int_s^T\big[\hat\Th_x(t,r,x)b(r,x)+{1\over 2}\si(r,x)^\prime\hat\Th_{xx}(t,r,x)\si(r,x)\\
&\qq\qq\qq\qq +g(t,r,x,\hat\Th(t,r,x),\hat\Th(r,r,x),\hat\Th_x(t,r,x)\si(r,x))\big]dr.
\end{align*}
Hence, $\hat\Th(t,s,x)\in C^{0,1,2}(\D[0,T]\times\dbR^d;\dbR^m)$ and satisfies the PDE \rf{FBTPDE1}.
Further, by \autoref{PDE-REP-BSVIE} and the uniqueness of the adapted solution to  EBSVIE \rf{mebsvie},
$\hat\Th(t,s,x)\in C^{0,1,2}(\D[0,T]\times\dbR^d;\dbR^m)$ is the unique classical solution to PDE \rf{FBTPDE1}.
\end{proof}
\begin{remark}\rm
Since the coefficient $\si(t,x)$ in \autoref{BSVIE-REP-PDE} is allowed to be degenerate and have a linear growth in $x$,
unlike the \cite[Theorem 5.2]{Wang--Yong 2018},  $\si(t,x)$  is not necessary to be uniformly positive
and bounded.
\end{remark}
\begin{remark}\rm
By \autoref{PDE-REP-BSVIE},  the formula \rf{PDE-REP-BSVIE-re}  gives a representation of the adapted solution to EBSVIE  \rf{mebsvie} via the classical solution to PDEs \rf{FBTPDE1}.
By \autoref{BSVIE-REP-PDE}, the formula \rf{hatTh} gives the probabilistic representation of the classical solution to non-local PDEs \rf{FBTPDE1}.
Thus, we generalize the nonlinear Feynman-Kac formula in Pardoux--Peng \cite{Pardoux--Peng 1992}.
\end{remark}


\end{document}